\documentclass[11pt]{article}

\usepackage{amsmath,amssymb,amsthm,mathtools,amsrefs,verbatim}
\usepackage{enumitem}
\usepackage[margin=1in]{geometry}
\usepackage{hyperref}
\usepackage{graphicx}
\usepackage{subcaption} 
\usepackage{calc}
\graphicspath{{figs/}}

\newcommand{\incfig}[3]{%
  \includegraphics[clip,trim=#3 #3 0mm 0mm,width=#2]{#1}%
}
\newcommand{\incfigl}[4]{%
  \includegraphics[clip,trim=#3 #4 0mm 0mm,width=#2]{#1}%
}

\title{Jellyfish exist}
\author{Ben Andrews\and Glen Wheeler}
\date{\today}

\theoremstyle{plain}
\newtheorem{theorem}{Theorem}[section]
\newtheorem{conj}[theorem]{Conjecture}
\newtheorem{proposition}[theorem]{Proposition}
\newtheorem{lemma}[theorem]{Lemma}
\newtheorem{corollary}[theorem]{Corollary}
\newtheorem{remark}[theorem]{Remark}
\newtheorem{definition}[theorem]{Definition}

\newcommand{\R}{\mathbb{R}}
\newcommand{\Q}{\mathbb{Q}}
\newcommand{\Z}{\mathbb{Z}}
\newcommand{\N}{\mathbb{N}}
\renewcommand{\S}{\mathbb{S}}
\newcommand{\C}{\mathbb{C}}
\newcommand{\SL}{\mathcal{L}}
\newcommand{\SE}{\mathcal{E}}
\newcommand{\SI}{\mathcal{I}}
\newcommand{\SF}{\mathcal{F}}
\newcommand{\SQ}{\mathcal{Q}}

\begin{document}
\maketitle

\begin{abstract}
We show the existence of infinitely many geometrically distinct homothetic expanders (jellyfish) for the elastic flow, epicyclic shrinkers for the curve diffusion flow, and epicyclic expanders for the ideal flow.
\end{abstract}

\begin{figure}[t]
\hfill
\centering
\begin{subfigure}[t]{0.32\textwidth}
  \centering
  \includegraphics[width=\linewidth]{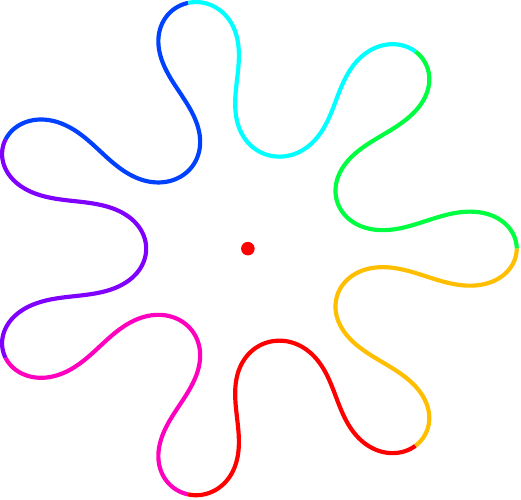}
  \caption{Free elastic flow jellyfish expander ($p/q=1/7$, $\omega=1$, $\varepsilon=0.13290$).}
  \label{fig:intro:rep:jellyfish}
\end{subfigure}\hfill
\begin{subfigure}[t]{0.32\textwidth}
  \centering
  \incfig{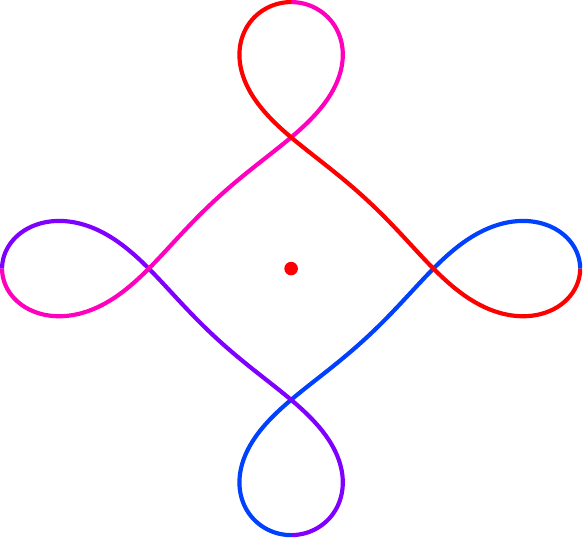}{\linewidth}{0mm}
  \caption{CDF epicyclic shrinker ($p/q=3/4$, $\omega=3$, $\varepsilon=0.20132$).}
  \label{fig:intro:rep:cdf}
\end{subfigure}\hfill
\begin{subfigure}[t]{0.32\textwidth}
  \centering
  \incfig{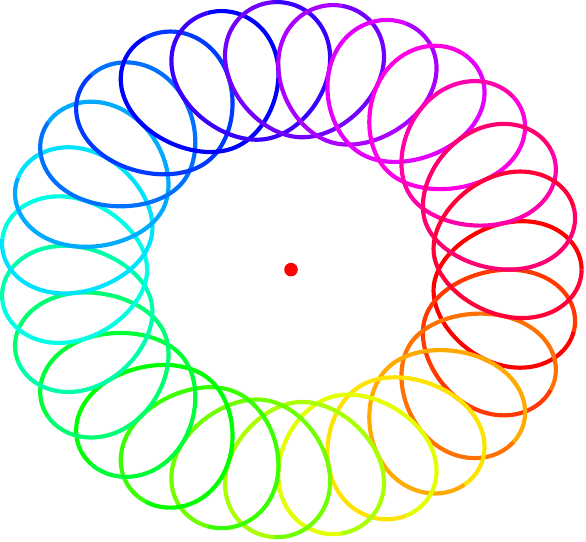}{\linewidth}{0mm}
  \caption{Ideal-flow epicyclic expander ($p/q=26/27$, $\omega=26$, $\varepsilon=0.20005$).}
  \label{fig:intro:rep:ideal}
\end{subfigure}
\hfill
\caption{Representative closed self-similar solutions from each family constructed in this paper.
Values of $\varepsilon$ are accurate to five significant figures.
Colours are used to denote doubled fundamental arcs.}
\label{fig:intro:representatives}
\end{figure}

\section{Introduction}

\subsection{Three curvature flow}

Let $\gamma:\S\times[0,T)\to\R^2$ be a one-parameter family of smooth, closed, immersed planar curves, where $T\in(0,\infty]$ is the maximal time of existence.
We write $s$ for arc-length along $\gamma(\cdot,t)$, $\tau=\partial_s\gamma$ for the unit tangent, $\nu$ for the unit normal, $k$ for the scalar curvature (so that $\partial_s\tau=k\nu$), and $k_s=\partial_s k$, $k_{ss}=\partial_s^2 k$, etc.
A planar curve's $\gamma$ associated  length, elastic energy, and ideal energy are as follows
\begin{equation}\label{eq:intro:functionals}
\SL[\gamma]=\int_{\gamma} ds,
\qquad
\SE[\gamma]=\frac12\int_{\gamma} k^2\,ds,
\qquad
\SI[\gamma]=\frac12\int_{\gamma} k_s^2\,ds.
\end{equation}

The steepest descent $L^2(ds)$-gradient flows of $\mathcal E$ and $\mathcal I$ are, respectively, the \emph{(free) elastic flow} and the \emph{ideal flow}.
The \emph{curve diffusion flow} is the steepest descent $H^{-1}(ds)$ (normal graphical) gradient flow of length.
If $\gamma$ is an elastic flow, its normal velocity satisfies
\begin{equation}\label{EF}\tag{EF}
\partial_t\gamma
= - \nabla^{L^2(ds)}\SE[\gamma]\nu
=
-\Big(k_{ss}+\frac12 k^3\Big)\nu,
\end{equation}
 if $\gamma$ is an ideal flow, then its normal velocity satisfies
\begin{equation}\label{IF}\tag{IF}
\partial_t\gamma
= - \nabla^{L^2(ds)}\SI[\gamma]\nu
=
\Big(k_{ssss}+k^2k_{ss}-\frac12 k\,k_s^2\Big)\nu,
\end{equation}
and finally, if $\gamma$ is a curve diffusion flow, its normal velocity satisfies
\begin{equation}\label{CDF}\tag{CDF}
\partial_t\gamma
=- \nabla^{H^{-1}(ds)}\SL[\gamma]
=-k_{ss}\,\nu.
\end{equation}
We typically drop the decorations and the argument on $\nabla$, denoting $\nabla\SE := \nabla^{L^2(ds)}\SE[\gamma]$, $\nabla\SI := \nabla^{L^2(ds)}\SI[\gamma]$ and $\nabla\SL = \nabla^{H^{-1}(ds)}\SL[\gamma]$.

All solutions to \eqref{EF} and \eqref{IF} are immortal (see \cite{AMWW20} for the ideal flow, which also contains a stability of circles result, and \cite{DKS02} for the elastic flow; for the stability of expanding circles, see \cite{MW25} and \cite{AW25}), and are expected to exhibit expansion behaviour.
On the other hand, solutions to \eqref{CDF} may exist only for finite time, and are expected to exhibit shrinking behaviour (see \cite{EGMWW15} for the lemniscate, \cite{EG97} and \cite{W13} for stability of circles (eventual convergence of immortal trajectories to multiply-covered circles holds, see \cite{W22} and \cite{M24}) and \cite{W13} for an estimate on the maximal existence time).
This can be understood as a consequence of the effect scaling action has on the energy: If $\eta = \rho\gamma$, then $\SL[\eta] = \rho\SL[\gamma]$, $\SE[\eta] = \rho^{-1}\SE[\gamma]$, and $\SI[\eta] = \rho^{-3}\SI[\gamma]$.
This suggests that the curve diffusion flow wishes to take $\rho\searrow0$ and that the elastic and ideal flows wish to take $\rho\nearrow\infty$.
For solutions evolving purely by homothety, that is solutions of the form \begin{equation}
\label{eq:intro:homot}
\gamma(\cdot,t) = \rho(t)\hat\gamma(\cdot), 
\text{ where for some $\sigma\in\R$ }
\begin{cases}
\nabla\SE[\hat\gamma] = -\sigma \hat\gamma\cdot\nu\,,\text{ for \eqref{EF}}
\\
\nabla\SI[\hat\gamma] = -\sigma \hat\gamma\cdot\nu\,,\text{ for \eqref{IF}}
\\
\nabla\SL[\hat\gamma] = -\sigma \hat\gamma\cdot\nu\,,\text{ for \eqref{CDF}}
\end{cases}
\end{equation}
this is exactly what does happen (and in finite time for the curve diffusion flow, infinite time for the elastic and ideal flows).
Non-stationary ($r(t)$ not constant) examples of profiles $\hat\gamma$ satisfying \eqref{eq:intro:homot}, outside of circles and Bernoulli's lemniscate, have yet to be found. 
The goal of this paper is to address this shortcoming.

\subsection{Main results}

\subsection*{Fundamental arcs and dihedral gluing}

The overarching strategy for the three families of solutions we construct is as follows.

\begin{itemize}
\item We begin with a {stationary base arc} for the relevant profile problem: a half-period of Euler's rectangular elastica for jellyfish, and a semicircle for the epicyclic family.
\item We introduce a small perturbation parameter and pose a boundary-value problem for a \emph{fundamental arc} which is compatible with a dihedral reflection principle (so that reflected concatenations are smooth across seams).
\item We prove existence of sufficiently many fundamental arcs (depending smoothly on parameters) and then impose a discrete dihedral closing condition ensuring that finitely many copies glue to a closed curve.
\end{itemize}

This turns the search for closed homothetic solutions into a combination of (i) a perturbative existence theory for an ODE boundary-value problem on an interval; and (ii) a symmetry-driven closing condition.

\smallskip
\noindent\textbf{Jellyfish expanders for the free elastic flow.}
These are closed expanders for \eqref{EF} with dihedral symmetry, visually resembling a ``body'' with multiple ``tentacles''.
They are constructed from a {fundamental arc} obtained by perturbing a half-period of Euler's rectangular elastica, and then gluing via reflections and rotations.

Specifically, let us denote by $\gamma^j_m,\gamma^e_m,\gamma^E_m:\S\to\R^2$ smooth closed curves with the following properties:
\begin{itemize}
\item $\gamma^j_m,\gamma^e_m,\gamma^E_m$ have dihedral symmetry of order $m$
\item $\gamma^j_m,\gamma^e_m,\gamma^E_m$ satisfy the homothetic solution equation \eqref{eq:intro:homot} for the elastic flow (and expands self-similarly), the curve diffusion flow (and shrinks self-similarly), and the ideal flow (and expands self-similarly) respectively
\item $\gamma^j_m,\gamma^e_m,\gamma^E_m$ is not equivalent to $\gamma^j_{\hat m},\gamma^e_{\hat m},\gamma^E_{\hat m}$ for $\hat m \ne m$ by similarity transformation.
\end{itemize}
We call $\gamma^j_m$ \emph{jellyfish}, $\gamma^e_m$ \emph{epicyclic shrinkers} and $\gamma^E_m$ \emph{epicyclic expanders}.

\begin{theorem}
\label{thm:intro:jellyfish}
There is an $m_0<\infty$ such that jellyfish expanders $\gamma^j_m$ exist for all $m>m_0$.
\end{theorem}

\smallskip
\noindent\textbf{Epicyclic shrinkers/expanders for curve diffusion and ideal flows.}
In this case, the fundamental arc is a semicircle.
Geometrically, the resulting closed curves resemble epicycles, with many ``petals''; analytically, they are homothetic shrinkers for \eqref{CDF} and homothetic expanders for \eqref{IF}.

A key technical distinction from the jellyfish construction is the presence of a trivial circle branch at the semicircle: a na\"ive implicit-function setup degenerates (and fails) because the gluing conditions are automatically satisfied by circles.
To recover a non-degenerate bifurcation problem, we normalise the matching conditions by an explicit division procedure which removes the trivial branch and isolates the genuinely new solutions.

\begin{theorem}\label{thm:intro:cdf-epi}
There is an $m_0<\infty$ such that epicyclic shrinkers $\gamma^e_m$ exist for all $m>m_0$.
\end{theorem}

\begin{theorem}\label{thm:intro:if-epi}
There is an $m_0<\infty$ such that epicyclic expanders $\gamma^E_m$ exist for all $m>m_0$.
\end{theorem}

Theorems~\ref{thm:intro:jellyfish}-\ref{thm:intro:if-epi} should be viewed as existence results for {infinite} families of new homothetic solutions, rather than isolated examples.
We prove these theorems in Sections 2, 3 and 4 of the paper.

\begin{figure}[t]
\hfill
\centering
\begin{subfigure}[t]{0.49\textwidth}
  \centering
  \includegraphics[angle=90,width=0.7\linewidth]{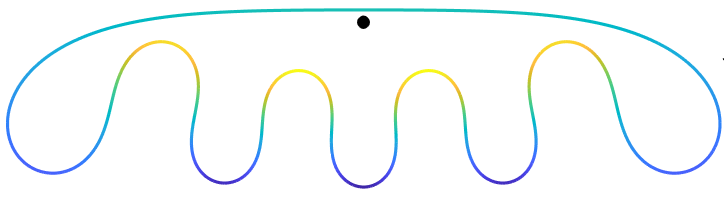}
\end{subfigure}\hfill
\begin{subfigure}[t]{0.49\textwidth}
  \centering
  \includegraphics[angle=90,width=0.7\linewidth]{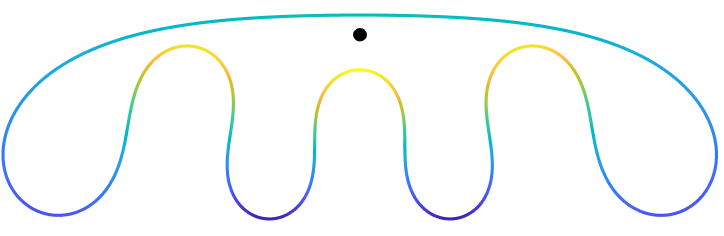}
\end{subfigure}
\hfill
\caption{Closed jellyfish-like expanders for the free elastic flow that do not possess dihedral symmetry. The indicated point is the centre of expansion.}
\label{fig:intro:mysterious}
\end{figure}

\subsection*{Open questions and outlook}
The families produced here do not come close to exhausting the space of homothetic solutions for any of the flows \eqref{EF}, \eqref{IF}, \eqref{CDF}.
Numerically we observe additional exotic solutions, including non-symmetric ``one-sided'' jellyfish (see Figure \ref{fig:intro:mysterious}), for which our method used here can not possibly apply.
Even restricting to a specific turning number appears to be a daunting problem.
A natural goal is to establish rigorous existence of such solutions, and, more ambitiously and more longer-term, to classify homothetic solutions for these  flows completely.

Beyond classification, the existence of large collections of expanders for \eqref{EF}, \eqref{CDF} and \eqref{IF} suggests a substantially richer dynamical picture than previously anticipated.
A first step toward understanding this complexity is to identify (even partially) the basins of attraction of the epicyclic expanders and jellyfish $\gamma^j_m,\gamma^e_m,\gamma^E_m$.
We suspect that these solutions are not stable for their respective flows, and venture the following conjecture.

\begin{conj}
The only dynamically stable solutions
\begin{enumerate}[label=(\alph*)]
\item to the curve diffusion flow are circles;
\item to the elastic flow are multiply-covered lemniscates of Bernoulli and multiply-covered circles;
\item and to the ideal flow are multiply-covered ideal lemniscates and multiply-covered circles.
\end{enumerate}
\end{conj}

Here a curve $\gamma$ is called \emph{dynamically stable} if there exists an $\varepsilon>0$ such that all curves $\eta$ with $||\gamma-\eta||_{C^\infty} < \varepsilon$ generate solution trajectories that converge, modulo similarity transformation, to $\gamma$.

This conjecture has as a precondition the discovery of an \emph{ideal lemniscate}, a self-similarly expanding figure-8 type solution to the ideal flow, which is yet to be established rigorously.

The natural next target once this conjecture is established would be the following analogue of \cite{W22} for the elastic and ideal flows.

\begin{conj}
Suppose $\gamma:[0,\infty)\to\R^2$ is an immortal elastic or ideal flow with generic initial data.
Then $\gamma(\cdot,t)$ converges  exponentially fast in the smooth topology to a multiply-covered lemniscate or circle.
\end{conj}

\section*{Acknowledgements}

The research of Ben Andrews was supported by grants FL150100126 and DP250103952 of the Australian Research Council.  
The research of Glen Wheeler was supported by grant DP250101080.
Work on the project was additionally facilitated by the Australian National University.

\section{Jellyfish expanders}
\label{sec:jellyfish-expanders}

In this section we prove Theorem~\ref{thm:intro:jellyfish} by constructing, for all sufficiently large
$m\in\N$, a closed homothetic expander for the free elastic flow \eqref{EF} with dihedral symmetry of order $m$.
The construction proceeds in two stages:
(i) an application of the Implicit Function Theorem produces a one-parameter family of {fundamental arcs}
as perturbations of a half-period of Euler's rectangular elastica; and
(ii) a dihedral reflection-rotation gluing produces a smooth closed curve (the {jellyfish}) from a fundamental arc.
Throughout we work in arc-length gauge.

\subsection{Existence of a fundamental arc}
\label{subsec:fundamental-arc-expander}

\subsubsection{ODE formulation and boundary map}

Fix parameters $(\alpha,\varepsilon)$ in a neighbourhood of $(0,0)$ to be chosen.  We consider
\[
S(s;\alpha,\varepsilon):=(x(s),y(s),\theta(s),k(s),v(s))\in\R^5
\]
solving
\begin{equation}\label{eq:ode-S}
\left\{
\begin{aligned}
x'&=-\sin\theta,\\
y'&=\phantom{-}\cos\theta,\\
\theta'&=k,\\
k'&=v,\\
v'&=-\frac12\,k^3-\alpha\cos\theta-\alpha\varepsilon\bigl(x\cos\theta+y\sin\theta\bigr),
\end{aligned}
\right.
\qquad
S(0;\alpha,\varepsilon)=(0,0,0,1,0).
\end{equation}
Define the boundary data map $B=(b_1,b_2)^{\mathsf T}$ by
\begin{equation}\label{eq:Bdef}
B(s,\alpha,\varepsilon)
:=
\begin{bmatrix}
b_1(s,\alpha,\varepsilon)\\
b_2(s,\alpha,\varepsilon)
\end{bmatrix}
:=
\begin{bmatrix}
-\sin\theta(s)+\varepsilon\bigl(y(s)\cos\theta(s)-x(s)\sin\theta(s)\bigr)\\
v(s)
\end{bmatrix}.
\end{equation}

\begin{definition}[Fundamental arc]\label{def:fundamental-arc}
A \emph{fundamental arc} is a solution $S(\,\cdot\,;\alpha,\varepsilon)$ of \eqref{eq:ode-S} for which there exists
$L>0$ such that
\begin{equation}\label{eq:fund-arc-BC}
B(L,\alpha,\varepsilon)=(0,0).
\end{equation}
\end{definition}

\begin{remark}[Geometric meaning of $b_1$]\label{rem:b1-geometry}
For $\varepsilon\neq 0$ set $c_\varepsilon:=(-1/\varepsilon,0)$. Writing $\gamma=(x,y)$ and
$T=(-\sin\theta,\cos\theta)$, a direct computation gives
\[
b_1(s,\alpha,\varepsilon)=\varepsilon\,\big\langle \gamma(s)-c_\varepsilon,\,T(s)\big\rangle.
\]
Thus $b_1(L,\alpha,\varepsilon)=0$ is the orthogonality condition
$\langle \gamma(L)-c_\varepsilon,\,T(L)\rangle=0$, which is precisely the condition that $\gamma$ meets the
ray through $c_\varepsilon$ and $\gamma(L)$ orthogonally. The condition $b_2(L,\alpha,\varepsilon)=v(L)=k_s(L)=0$
is needed for smooth reflection across that ray.
\end{remark}

The follow is a consequence of standard ODE theory.

\begin{lemma}[Smooth dependence]\label{lem:smooth-dependence}
There exist $\eta>0$ and a neighbourhood $U\subset\R^2$ of $(0,0)$ such that for all $(\alpha,\varepsilon)\in U$
the solution $S(\,\cdot\,;\alpha,\varepsilon)$ exists on $[0,L_0+\eta]$ (where $L_0$ is fixed below), and the map
\[
(s,\alpha,\varepsilon)\longmapsto S(s;\alpha,\varepsilon)
\]
is $C^\infty$ (indeed real-analytic) on $[0,L_0+\eta]\times U$. Consequently $B$ is $C^\infty$ on this set.
\end{lemma}

\subsubsection{The base arc at $(\alpha,\varepsilon)=(0,0)$}

Setting $\alpha=0$ in \eqref{eq:ode-S}, the curvature subsystem closes:
\begin{equation}\label{eq:base-curv}
k''+\frac12\,k^3=0,\qquad k(0)=1,\qquad k'(0)=v(0)=0.
\end{equation}

In order to streamline our presentation, we note the following three integrations (one scalar and one two-dimensional vector) of the curvature subsystem.

\begin{lemma}
\label{lem:base-first-integral}
The solution of \eqref{eq:base-curv} satisfies the identity
\begin{equation}\label{eq:energy-base}
\frac12 (k')^2+\frac18 k^4=\frac18
\quad\Longleftrightarrow\quad
v^2=\frac14(1-k^4).
\end{equation}
In particular $|k|\le 1$ and $k$ oscillates periodically between $1$ and $-1$.
\end{lemma}

\begin{proof}
Multiply \eqref{eq:base-curv} by $k'$ and integrate in $s$, using $k'(0)=0$ and $k(0)=1$.
\end{proof}

\begin{lemma}\label{lem:force-identities}
Define
\begin{equation}\label{eq:force-def}
\mathcal F := \frac12 k^2\,T + k'\,N,
\qquad
T:=(-\sin\theta,\cos\theta),\quad N:=(-\cos\theta,-\sin\theta).
\end{equation}
Then
\begin{equation}\label{eq:force-deriv-general}
\mathcal F'
=\Bigl(k''+\frac12 k^3\Bigr)N.
\end{equation}
In particular, along the base solution $(\alpha,\varepsilon)=(0,0)$ we have $\mathcal F\equiv (0,\frac12)$, and
\begin{equation}\label{eq:cos-sin-identities}
\cos\theta = k^2,\qquad
\sin\theta = -2v.
\end{equation}
\end{lemma}

\begin{proof}
Using $T'=kN$ and $N'=-kT$, a direct calculation yields \eqref{eq:force-deriv-general}. Due to \eqref{eq:base-curv}, $\mathcal F'\equiv 0$, so $\mathcal F$ is a constant vector. Evaluating at $s=0$ gives
$T(0)=(0,1)$, $N(0)=(-1,0)$, $k(0)=1$ and $k'(0)=0$, hence $\mathcal F\equiv(0,\frac12)$.
Writing $\mathcal F=(\mathcal F_x,\mathcal F_y)$ yields
\begin{align}
0=\mathcal F_x
&=-\frac12 k^2\sin\theta - v\cos\theta,
\label{eq:Fx-zero}\\
\frac12=\mathcal F_y
&=\frac12 k^2\cos\theta - v\sin\theta.
\label{eq:Fy-half}
\end{align}
Eliminate $\sin\theta$ using \eqref{eq:Fx-zero} and then use \eqref{eq:energy-base} to obtain
$\cos\theta=k^2$, and hence $\sin\theta=-2v$.
\end{proof}

We conclude as follows.

\begin{lemma}\label{lem:base-halfperiod}
Let $(x_0,y_0,\theta_0,k_0,v_0)$ denote the solution of \eqref{eq:ode-S} at $(\alpha,\varepsilon)=(0,0)$.
Then there exists a unique $L_0\in(0,\infty)$ such that
\begin{equation}\label{eq:L0-def}
L_0=\inf\{\,s>0:\ k_0(s)=-1\,\}.
\end{equation}
Moreover, $k_0'(L_0)=v_0(L_0)=0$ and $\theta_0(L_0)=0$. In particular,
\begin{equation}\label{eq:B-base-zero}
B(L_0,0,0)=\bigl(-\sin\theta_0(L_0),\,v_0(L_0)\bigr)=(0,0).
\end{equation}
\end{lemma}

\begin{proof}
The base curvature satisfies \eqref{eq:base-curv} and the first integral \eqref{eq:energy-base}, namely
\[
\bigl(k_0'(s)\bigr)^2=\frac14\bigl(1-k_0(s)^4\bigr).
\]
Since $k_0''(0)=-\frac12k_0(0)^3=-\frac12<0$ and $k_0'(0)=0$, we have $k_0'(s)<0$ for all sufficiently small $s>0$,
and hence $k_0$ decreases from $k_0(0)=1$.
On any interval where $|k_0|<1$, the first integral forces
\[
k_0'(s)=-\frac12\sqrt{1-k_0(s)^4},
\]
so $k_0$ is strictly decreasing there. Separating variables yields the bound
\[
s
= \int_{0}^{s} 1\,d\tau
= \int_{k_0(s)}^{1}\frac{-2\,d\kappa}{\sqrt{1-\kappa^{4}}}
\le 2\int_{-1}^{1}\frac{d\kappa}{\sqrt{1-\kappa^{4}}}
<\infty,
\]
which implies that $k_0$ reaches the value $-1$ in finite time. (If not, then $s$ would be bounded, which means the ODE could not be continued, which is a contradiction.)
Consequently the set in \eqref{eq:L0-def} is nonempty, so $L_0\in(0,\infty)$ is well-defined.

At $s=L_0$ we have $k_0(L_0)=-1$, and thus \eqref{eq:energy-base} gives $k_0'(L_0)=0$, i.e.\ $v_0(L_0)=0$.
By Lemma~\ref{lem:force-identities} (at $(\alpha,\varepsilon)=(0,0)$) we have the identities
\[
\cos\theta_0(s)=k_0(s)^2\ge 0,\qquad \sin\theta_0(s)=-2v_0(s),
\]
hence $\cos\theta_0(L_0)=1$ and $\sin\theta_0(L_0)=0$, so $\theta_0(L_0)\in 2\pi\Z$.
Since $\theta_0(0)=0$ and $\cos\theta_0(s)\ge 0$ for all $s\in[0,L_0]$, the continuous map $\theta_0$ cannot cross beyond $\pm\pi/2$ on $[0,L_0]$.
This forces $\theta_0(L_0)=0$.

Finally, at $(\alpha,\varepsilon)=(0,0)$ one has $B=( -\sin\theta_0,\ v_0 )^{\mathsf T}$ by \eqref{eq:Bdef}, and therefore
$B(L_0,0,0)=(0,0)$ follows from $\sin\theta_0(L_0)=0$ and $v_0(L_0)=0$.
\end{proof}

\subsubsection{Solving $b_1=0$ for $L=L(\alpha,\varepsilon)$}

The following non-degeneracy in $s$ is enough to apply the Implicit Function Theorem.

\begin{lemma}\label{lem:b1-s-nondeg}
The derivative $\partial_s b_1(L_0,0,0)$ is nonzero. More precisely,
\begin{equation}\label{eq:b1s-base}
\partial_s b_1(L_0,0,0)=1.
\end{equation}
\end{lemma}

\begin{proof}
Differentiate \eqref{eq:Bdef} in $s$ to obtain
\begin{align}
\partial_s b_1
&=
-(\cos\theta)\,\theta'
+\varepsilon\Bigl(
y'\cos\theta - y\sin\theta\,\theta'
-x'\sin\theta - x\cos\theta\,\theta'
\Bigr)\nonumber\\
&=
-k\cos\theta
+\varepsilon\Bigl(\cos^2\theta+\sin^2\theta-k(x\cos\theta+y\sin\theta)\Bigr)\nonumber\\
&=
-k\cos\theta+\varepsilon\Bigl(1-k(x\cos\theta+y\sin\theta)\Bigr).
\label{eq:b1s}
\end{align}
At $(L_0,0,0)$ we have $\varepsilon=0$, $\cos\theta(L_0;0,0)=1$ and $k(L_0;0,0)=-1$, so $\partial_s b_1(L_0,0,0)=1$.
\end{proof}

\begin{proposition}\label{prop:L-of-alpha-eps}
There exist $\delta_1>0$ and a unique $C^\infty$ map
\[
(\alpha,\varepsilon)\mapsto L(\alpha,\varepsilon)\in(L_0-\eta,L_0+\eta),
\qquad |\alpha|+|\varepsilon|<\delta_1,
\]
such that
\begin{equation}\label{eq:Lsolvesb1}
b_1\bigl(L(\alpha,\varepsilon),\alpha,\varepsilon\bigr)=0
\quad\text{and}\quad
L(0,0)=L_0.
\end{equation}
\end{proposition}

\begin{proof}
Apply the Implicit Function Theorem to the equation $b_1(s,\alpha,\varepsilon)=0$ at $(s,\alpha,\varepsilon)=(L_0,0,0)$,
using Lemma~\ref{lem:b1-s-nondeg}.
\end{proof}

\subsubsection{The fundamental arc}

Define the scalar shooting function
\begin{equation}\label{eq:gdef}
g(\alpha,\varepsilon):=b_2\bigl(L(\alpha,\varepsilon),\alpha,\varepsilon\bigr)
=v\bigl(L(\alpha,\varepsilon);\alpha,\varepsilon\bigr).
\end{equation}
Then $g$ is $C^\infty$ and $g(0,0)=0$ by \eqref{eq:B-base-zero} and \eqref{eq:Lsolvesb1}.

First we calculate the linearisation of $b_1$ in the $\alpha$-direction.

\begin{lemma}\label{lem:db1dalpha-base}
Let $b_1$ be defined by \eqref{eq:Bdef}, and let $S_0=(x_0,y_0,\theta_0,k_0,v_0)$ denote the base solution
of \eqref{eq:ode-S} at $(\alpha,\varepsilon)=(0,0)$. Let $L_0$ be as in Lemma~\ref{lem:base-halfperiod}.
Then
\begin{equation}\label{eq:db1dalpha-base}
\partial_\alpha b_1(L_0,0,0) = -\,\partial_\alpha\theta(L_0;0,0).
\end{equation}
More generally, for every $s$ in the common existence interval one has
\begin{equation}\label{eq:db1dalpha-eps0}
\partial_\alpha b_1(s,0,0) = -\,\cos\!\bigl(\theta_0(s)\bigr)\,\partial_\alpha\theta(s;0,0).
\end{equation}
\end{lemma}

\begin{proof}
Fix $s$ and set $\varepsilon=0$ in \eqref{eq:Bdef}.
The $\varepsilon$-term vanishes identically (for all $\alpha$), and we obtain
\[
\partial_\alpha b_1(s,0,0)
=
-\cos\!\bigl(\theta(s;0,0)\bigr)\,\partial_\alpha\theta(s;0,0)
=
-\cos\!\bigl(\theta_0(s)\bigr)\,\partial_\alpha\theta(s;0,0),
\]
which is \eqref{eq:db1dalpha-eps0}. Finally, at $s=L_0$ we have $\theta_0(L_0)=0$ by Lemma~\ref{lem:base-halfperiod},
so $\cos(\theta_0(L_0))=1$, and \eqref{eq:db1dalpha-base} follows.
\end{proof}

The non-degeneracy we require is the following.

\begin{lemma}\label{lem:galpha-nondeg}
One has
\begin{equation}\label{eq:galpha-equals-key}
\partial_\alpha g(0,0)
=
-\int_0^{L_0}\cos^2\theta_0(s)\,ds
\;<\;0,
\end{equation}
where $\theta_0$ denotes the base solution at $(\alpha,\varepsilon)=(0,0)$.
\end{lemma}

\begin{proof}
Define $\mathcal F$ by \eqref{eq:force-def} for general $(\alpha,\varepsilon)$. From \eqref{eq:ode-S} we have
\[
k''+\frac12k^3
=
v'+\frac12k^3
=
-\alpha\cos\theta-\alpha\varepsilon(x\cos\theta+y\sin\theta),
\]
hence by \eqref{eq:force-deriv-general},
\begin{equation}\label{eq:force-with-alpha}
\mathcal F'
=
-\alpha\Bigl(\cos\theta+\varepsilon(x\cos\theta+y\sin\theta)\Bigr)N.
\end{equation}
Differentiate \eqref{eq:force-with-alpha} with respect to $\alpha$ and then set $(\alpha,\varepsilon)=(0,0)$.
Because the parenthetical factor is multiplied by $\alpha$, all $\alpha$-derivatives of $x,y,\theta$ vanish at $\alpha=0$,
and we obtain
\begin{equation}\label{eq:force-linearised}
\partial_s\bigl(\partial_\alpha \mathcal F\bigr)(s;0,0)
=
-(\cos\theta_0(s))\,N_0(s),
\end{equation}
with $\partial_\alpha\mathcal F(0;0,0)=0$ (the initial data are independent of $\alpha$).
Integrating \eqref{eq:force-linearised} from $0$ to $L_0$ gives
\[
\partial_\alpha \mathcal F(L_0;0,0)
=
-\int_0^{L_0} \cos\theta_0(s)\,N_0(s)\,ds.
\]
Taking the $x$-component and using $N_x=-\cos\theta$ yields
\begin{equation}\label{eq:dFxdalpha-positive}
\partial_\alpha \mathcal F_x(L_0;0,0)
=
\int_0^{L_0} \cos^2\theta_0(s)\,ds
>0.
\end{equation}
On the other hand, $\mathcal F_x=-(1/2)k^2\sin\theta - v\cos\theta$, and at $s=L_0$ we have
$\sin\theta_0(L_0)=0$, $v_0(L_0)=0$, $\cos\theta_0(L_0)=1$, $k_0(L_0)^2=1$. Thus
\begin{equation}\label{eq:dFxdalpha-endpoint}
\partial_\alpha \mathcal F_x(L_0;0,0)
=
-\frac12\,\partial_\alpha\theta(L_0;0,0) \;-\;\partial_\alpha v(L_0;0,0).
\end{equation}
Combining \eqref{eq:dFxdalpha-positive}-\eqref{eq:dFxdalpha-endpoint} yields
\begin{equation}\label{eq:key-nondeg}
\partial_\alpha v(L_0;0,0)+\frac12\,\partial_\alpha\theta(L_0;0,0)
=
-\int_0^{L_0}\cos^2\theta_0(s)\,ds
\;<\;0.
\end{equation}
Finally, by \eqref{eq:gdef} and the chain rule,
\[
\partial_\alpha g(0,0)
=
\partial_\alpha v(L_0;0,0)
+\partial_s v(L_0;0,0)\,\partial_\alpha L(0,0).
\]
Differentiate $b_1(L(\alpha,\varepsilon),\alpha,\varepsilon)\equiv 0$ in $\alpha$ at $(0,0)$: since
$\partial_\alpha b_1(L_0,0,0)=-\partial_\alpha\theta(L_0;0,0)$ (Lemma \ref{lem:db1dalpha-base}), and $\partial_s b_1(L_0,0,0)=1$ by \eqref{eq:b1s-base}, we find $\partial_\alpha L(0,0)=\partial_\alpha\theta(L_0;0,0)$.
Moreover, from \eqref{eq:ode-S} at $(0,0)$,
\[
\partial_s v(L_0;0,0)=v'(L_0;0,0)=-\frac12\,k_0(L_0)^3=+\frac12,
\qquad\text{since }k_0(L_0)=-1.
\]
Therefore
\[
\partial_\alpha g(0,0)
=
\partial_\alpha v(L_0;0,0)+\frac12\,\partial_\alpha\theta(L_0;0,0),
\]
and \eqref{eq:galpha-equals-key} follows from \eqref{eq:key-nondeg}.
\end{proof}

We thus obtain existence of fundamental arcs.

\begin{proposition}\label{prop:fundamental-arc}
There exist $\delta\in(0,\delta_1)$ and unique $C^\infty$ functions
\[
\varepsilon\mapsto \alpha(\varepsilon),\qquad \varepsilon\mapsto L(\varepsilon),
\qquad |\varepsilon|<\delta,
\]
with $\alpha(0)=0$ and $L(0)=L_0$, such that the solution of \eqref{eq:ode-S} satisfies
\[
B\bigl(L(\varepsilon),\alpha(\varepsilon),\varepsilon\bigr)=(0,0)
\qquad\text{for all }|\varepsilon|<\delta.
\]
Equivalently, for each $|\varepsilon|<\delta$ the trajectory $S(\,\cdot\,;\alpha(\varepsilon),\varepsilon)$
is a fundamental arc in the sense of Definition~\ref{def:fundamental-arc}.
\end{proposition}

\begin{proof}
By Lemma~\ref{lem:galpha-nondeg} we have $\partial_\alpha g(0,0)\neq 0$. The Implicit Function Theorem applied to
$g(\alpha,\varepsilon)=0$ at $(0,0)$ yields a unique $C^\infty$ map $\alpha=\alpha(\varepsilon)$ for $|\varepsilon|<\delta$.
Define $L(\varepsilon):=L(\alpha(\varepsilon),\varepsilon)$ using Proposition~\ref{prop:L-of-alpha-eps}. Then
$b_1(L(\varepsilon),\alpha(\varepsilon),\varepsilon)=0$ and $v(L(\varepsilon);\alpha(\varepsilon),\varepsilon)=0$ by construction,
which is exactly $B(L(\varepsilon),\alpha(\varepsilon),\varepsilon)=(0,0)$.
\end{proof}

Let us now investigate the behaviour of the parameter $\alpha(\varepsilon)$ to first order near zero.
The below proves that $\alpha(\varepsilon)$ has the sign of $-\varepsilon$ for all $\varepsilon$ sufficiently close to $0$.

\begin{lemma}\label{lem:alpha-expansion-firstorder}
Let $g$ be the shooting function defined in \eqref{eq:gdef} and $\delta>0$, 
$\alpha:(-\delta,\delta)\to\R$ with $\alpha(0)=0$ be from Proposition \ref{prop:fundamental-arc}.
Then 
\begin{equation}\label{eq:alpha-firstorder}
\alpha(\varepsilon)=\alpha'(0)\,\varepsilon+O(\varepsilon^2)
\qquad(\varepsilon\to 0),
\end{equation}
where
\begin{equation}\label{eq:alpha-prime}
\alpha'(0)=-\frac{\partial_\varepsilon g(0,0)}{\partial_\alpha g(0,0)}
=-\frac{y_0(L_0)}{2\displaystyle\int_0^{L_0}\cos^2\theta_0(s)\,ds}<0.
\end{equation}

\end{lemma}

\begin{proof}
Since $\alpha$ is $C^\infty$
it admits the expansion \eqref{eq:alpha-firstorder}.
Differentiating 
\begin{equation}\label{eq:g-zero-identity}
g(\alpha(\varepsilon),\varepsilon)\equiv 0\qquad\text{for }|\varepsilon|<\delta.
\end{equation}
 at $\varepsilon=0$ gives
\[
\partial_\alpha g(0,0)\,\alpha'(0)+\partial_\varepsilon g(0,0)=0,
\]
so $\alpha'(0)=-\partial_\varepsilon g(0,0)/\partial_\alpha g(0,0)$.

To compute $\partial_\varepsilon g(0,0)$, recall
\[
g(\alpha,\varepsilon)=v\bigl(L(\alpha,\varepsilon);\alpha,\varepsilon\bigr),
\qquad
b_1\bigl(L(\alpha,\varepsilon),\alpha,\varepsilon\bigr)=0.
\]
At $\alpha=0$ the curvature subsystem reduces to \eqref{eq:base-curv}, hence
$v(\,\cdot\,;0,\varepsilon)\equiv v_0(\,\cdot\,)$ is independent of $\varepsilon$, and therefore
\[
\partial_\varepsilon g(0,0)=\partial_s v_0(L_0)\,\partial_\varepsilon L(0,0).
\]
Differentiate $b_1(L(0,\varepsilon),0,\varepsilon)\equiv 0$ at $\varepsilon=0$:
\[
0=\partial_s b_1(L_0,0,0)\,\partial_\varepsilon L(0,0)+\partial_\varepsilon b_1(L_0,0,0).
\]
Using \eqref{eq:b1s-base} gives $\partial_s b_1(L_0,0,0)=1$, and since $\theta_0(L_0)=0$ one has
$\partial_\varepsilon b_1(L_0,0,0)=y_0(L_0)$. Hence $\partial_\varepsilon L(0,0)=-y_0(L_0)$.

Finally, at $(\alpha,\varepsilon)=(0,0)$ we have $v_0'(s)=-\frac12 k_0(s)^3$, so
\[
\partial_s v_0(L_0)=v_0'(L_0)=+\frac12
\qquad\text{since }k_0(L_0)=-1,
\]
and therefore $\partial_\varepsilon g(0,0)=-\frac12\,y_0(L_0)$.

On the other hand, \eqref{eq:galpha-equals-key} yields $\partial_\alpha g(0,0)<0$, 
so \eqref{eq:alpha-prime} follows. Since $y_0(L_0)>0$ (because $y_0'(s)=\cos\theta_0(s)=k_0(s)^2\ge 0$ and $k_0\not\equiv 0$),
we have $\alpha'(0)<0$, and the sign conclusions follow from \eqref{eq:alpha-firstorder}.
\end{proof}

Set $c_\varepsilon:=(-1/\varepsilon,0)$ and $\sigma(\varepsilon):=-\varepsilon\,\alpha(\varepsilon)$.
These are the centre of the homothety and the homothety parameter, respectively.
Suppose that, for a given choice of sign of $\varepsilon$, the dihedral reflection-rotation construction of
Proposition~\ref{prop:dihedral-closure} below yields a smooth closed curve $\gamma_\varepsilon$.
Then, due to \eqref{eq:ode-S}, 
\begin{equation}\label{eq:homothety-geometric}
k_{ss}+\frac12 k^3 = -\sigma(\varepsilon)\,\langle \gamma_\varepsilon-c_\varepsilon,\,N\rangle.
\end{equation}
To see this, begin by computing 
\[
\cos\theta+\varepsilon(x\cos\theta+y\sin\theta)
= -\varepsilon\langle \gamma_\varepsilon-c_\varepsilon,\,N\rangle,
\]
so that the $v$-equation becomes
\[
v'=-\frac12 k^3-\alpha(\varepsilon)\bigl(\cos\theta+\varepsilon(x\cos\theta+y\sin\theta)\bigr)
= -\frac12 k^3+\alpha(\varepsilon)\varepsilon\langle \gamma_\varepsilon-c_\varepsilon,\,N\rangle,
\]
i.e.\ \eqref{eq:homothety-geometric} with $\sigma(\varepsilon)=-\varepsilon\alpha(\varepsilon)$ since $k_{ss}=v'$.

The elastic flow $\gamma(t)$ with $\gamma_\varepsilon-c_\varepsilon$ as initial data 
is self-similar, evolving according to \eqref{eq:intro:homot}.
Indeed, we claim 
\begin{equation}
\label{eq:sss}
\gamma(t) = (4\sigma(\varepsilon)t+1)^{\frac14}\gamma(0)
=: \lambda(t)\gamma(0)\,.
\end{equation}
This follows by direct differentiation and substitution of \eqref{eq:homothety-geometric}, keeping in mind that each of the terms on the LHS of \eqref{eq:homothety-geometric} are homogenous of degree $-3$.
Thus it follows from \eqref{eq:sss} that $\sigma(\varepsilon)>0$ implies $\gamma$ is expanding (called an \emph{expander}) and if $\sigma(\varepsilon)<0$ then $\gamma$  is shrinking (called a \emph{shrinker}). 

Along the elastic flow \eqref{EF}, the elastic energy
\[
\mathcal E[\gamma]:=\frac12\int_\gamma k^2\,ds
\]
satisfies  (note that as our solution is a smooth closed curve, $\SE[\gamma]$ is finite and differentiable) 
\begin{equation}\label{eq:energy-dissipation}
\frac{d}{dt}\mathcal E[\gamma(t)]=-\int_\gamma \bigl(k_{ss}+\tfrac12 k^3\bigr)^2\,ds\le 0,
\end{equation}
with equality if and only if the solution is stationary (up to tangential reparametrisation).

If $\gamma$ satisfies \eqref{eq:intro:homot} with parameter $\sigma$, then (due to the scaling relations $k(t)=\lambda(t)^{-1}k(0)$, $ds(t)=\lambda(t)\,ds(0)$, 
\[
\mathcal E[\gamma(t)]=\frac{\mathcal E[\gamma]}{\lambda(t)}
\quad\Longrightarrow\quad
\frac{d}{dt}\mathcal E[\gamma(t)]=-\frac{\lambda'(t)}{\lambda(t)^2}\,\mathcal E[\gamma].
\]
Since $\mathcal E[\gamma]>\frac{2\pi^2}{L[\gamma]} > 0$ for a closed curve (by Poincar\'e), we conclude by \eqref{eq:energy-dissipation}that $\lambda'(t)\ge 0$ for all $t$ for which the homothetic evolution exists.

From the above discussion, an important consistency check is that 
the homothety coefficient  $\sigma(\varepsilon)$ maintains a sign for $\varepsilon$ positive and negative.

\begin{lemma}\label{lem:sigma-positive}
Let $\alpha(\varepsilon)$ and $L(\varepsilon)$ be the functions from Proposition~\ref{prop:fundamental-arc}, so that
$B(L(\varepsilon),\alpha(\varepsilon),\varepsilon)=(0,0)$ and $\alpha(0)=0$, $L(0)=L_0$.
Then
\begin{equation}\label{eq:sigma-positive-expansion}
\sigma(\varepsilon)= -\alpha'(0)\,\varepsilon^2+O(\varepsilon^3)>0
\qquad\text{for all sufficiently small }\varepsilon\neq 0.
\end{equation}
\end{lemma}

\begin{proof}
Multiplying \eqref{eq:alpha-firstorder} by $-\varepsilon$ gives
\[
\sigma(\varepsilon)=-\varepsilon\alpha(\varepsilon)
= -\alpha'(0)\,\varepsilon^2+O(\varepsilon^3).
\]
Since $\alpha'(0)<0$, the leading coefficient $-\alpha'(0)$ is strictly positive, which yields
\eqref{eq:sigma-positive-expansion}.
\end{proof}

We generate fundamental arcs for positive and negative $\varepsilon$.
The above implies that each branch is expanding, as we expect.
It could nevertheless occur that these branches of solutions are geometrically distinct.
This turns out not to be the case; these branches are equivalent, and we may follow either of them.
The following results establish this.

\begin{lemma}[Reflection symmetry exchanging $\varepsilon$]\label{lem:eps-reflection-symmetry}
Let $R:\R^2\to\R^2$ be reflection in the $y$-axis, $R(x,y)=(-x,y)$.
Fix parameters $(\alpha,\varepsilon)$ and let
\[
S(s;\alpha,\varepsilon)=(x,y,\theta,k,v)
\]
solve \eqref{eq:ode-S} on an interval $[0,L]$.
Define the transformed state
\[
\widetilde S(s):=(\widetilde x,\widetilde y,\widetilde\theta,\widetilde k,\widetilde v)
:=(-x,\ y,\ -\theta,\ -k,\ -v).
\]
Then $\widetilde S$ solves \eqref{eq:ode-S} with parameters $(-\alpha,-\varepsilon)$, i.e.
\[
\widetilde S(\,\cdot\,)=S(\,\cdot\,;-\alpha,-\varepsilon)
\quad\text{up to the choice of initial curvature sign}.
\]
Moreover, the boundary map transforms by
\begin{equation}\label{eq:B-transforms}
B\bigl(s;-\alpha,-\varepsilon\bigr)\ \text{computed on }\widetilde S
=
\begin{bmatrix}
- b_1(s;\alpha,\varepsilon)\\
- b_2(s;\alpha,\varepsilon)
\end{bmatrix}.
\end{equation}
In particular, $B(L;\alpha,\varepsilon)=(0,0)$ if and only if $B(L;-\alpha,-\varepsilon)=(0,0)$ for the reflected arc.

Equivalently, in geometric form (for $\varepsilon\neq 0$), the homothety equation
\begin{equation}\label{eq:homothety-geometric-again}
k_{ss}+\frac12 k^3 = -\sigma\,\langle \gamma-c,\,N\rangle
\end{equation}
is invariant under reflection: $R$ sends a solution with centre $c$ to a solution with centre $R(c)$ and the {same}
coefficient $\sigma$.
\end{lemma}

\begin{proof}
A direct substitution shows that $(\widetilde x,\widetilde y,\widetilde\theta,\widetilde k,\widetilde v)$ satisfy
the first four equations of \eqref{eq:ode-S}. For the last equation, use
\[
\cos(\widetilde\theta)=\cos\theta,\qquad \sin(\widetilde\theta)=-\sin\theta,\qquad
\widetilde x\cos\widetilde\theta+\widetilde y\sin\widetilde\theta
=-(x\cos\theta+y\sin\theta),
\]
and $\widetilde k^3=-k^3$, $\widetilde v'=-v'$, to obtain precisely the $v$-equation with $(\alpha,\varepsilon)$ replaced by
$(-\alpha,-\varepsilon)$.

For the boundary map, compute
\[
-\sin\widetilde\theta+(-\varepsilon)\bigl(\widetilde y\cos\widetilde\theta-\widetilde x\sin\widetilde\theta\bigr)
=
-\sin(-\theta)-\varepsilon\bigl(y\cos\theta-(-x)\sin(-\theta)\bigr)
=
-\,\Bigl(-\sin\theta+\varepsilon(y\cos\theta-x\sin\theta)\Bigr),
\]
and $\widetilde v=-v$, giving \eqref{eq:B-transforms}. Zeros are therefore preserved.

Finally, for $\varepsilon\neq 0$, the centre is $c_\varepsilon=(-1/\varepsilon,0)$ and satisfies $R(c_\varepsilon)=c_{-\varepsilon}$.
Writing the forcing term as $-\sigma\langle \gamma-c_\varepsilon,N\rangle$ with $\sigma=-\varepsilon\alpha$ shows that reflection leaves
$\sigma$ unchanged and sends $c_\varepsilon$ to $c_{-\varepsilon}$, which is \eqref{eq:homothety-geometric-again}.
\end{proof}

\subsection{Abundance of seam angles}
\label{subsec:cdf:angles}

As in the jellyfish construction, the dihedral closure requires control of the angle between the two
{radial seam lines} through the homothety centre.  In the present CDF setting the base arc is a semicircle,
so this seam angle is close to $\pi$ (rather than close to $0$ as in the jellyfish case).  It is therefore
convenient to work with the {angle deficit from $\pi$}.

Along the IFT branch from Proposition~\ref{prop:cdf:fundamental-arc} define
\[
\Theta(\varepsilon):=\theta\bigl(L(\varepsilon);\alpha(\varepsilon),\varepsilon\bigr),
\qquad
\overline{\theta}(\varepsilon):=\pi-\Theta(\varepsilon).
\]
Thus $\overline{\theta}(\varepsilon)$ is the acute angle between the radial line through $c_\varepsilon$ and $\gamma(0)$ and the
radial line through $c_\varepsilon$ and $\gamma(L(\varepsilon))$.

\begin{lemma}[Angle deficit expansion]\label{lem:cdf:theta-bar-expansion}
Along the branch from Proposition~\ref{prop:cdf:fundamental-arc} one has
\begin{equation}\label{eq:cdf:theta-bar-expansion}
\overline{\theta}(\varepsilon)=\pi\,\varepsilon^2+O(\varepsilon^3)
\qquad(\varepsilon\to 0).
\end{equation}
In particular, $\overline{\theta}(\varepsilon)>0$ for all sufficiently small $\varepsilon\neq 0$.
Moreover,
\begin{equation}\label{eq:cdf:theta-bar-derivative}
\overline{\theta}'(\varepsilon)=2\pi\,\varepsilon+O(\varepsilon^2),
\qquad(\varepsilon\to 0),
\end{equation}
and hence $\overline{\theta}$ is strictly increasing on $(0,\varepsilon_\ast)$ and strictly decreasing on $(-\varepsilon_\ast,0)$
for some $\varepsilon_\ast>0$.
\end{lemma}

\begin{proof}
This is an immediate reformulation of Lemma~\ref{lem:cdf:Theta-expansion}, which asserts
$\Theta(\varepsilon)=\pi-\pi\varepsilon^2+O(\varepsilon^3)$, hence \eqref{eq:cdf:theta-bar-expansion}.
Differentiating the expansion gives \eqref{eq:cdf:theta-bar-derivative}, and the stated one-sided monotonicity follows.
\end{proof}

\begin{corollary}[One-sided reparametrisation by the deficit angle]\label{cor:cdf:eps-as-fn-of-theta}
There exist $\theta_\ast>0$ and $\varepsilon_\ast>0$ and unique $C^\infty$ functions
\[
\varepsilon=\varepsilon_\pm(\overline{\theta})
\qquad\text{for }\ \overline{\theta}\in(0,\theta_\ast),
\]
with $\varepsilon_+(\overline{\theta})\in(0,\varepsilon_\ast)$ and $\varepsilon_-(\overline{\theta})\in(-\varepsilon_\ast,0)$, such that
\[
\pi-\theta\bigl(L(\varepsilon_\pm(\overline{\theta}));\alpha(\varepsilon_\pm(\overline{\theta})),\varepsilon_\pm(\overline{\theta})\bigr)
=\overline{\theta}.
\]
Moreover,
\begin{equation}\label{eq:cdf:eps-asymp}
\varepsilon_\pm(\overline{\theta})
=
\pm\sqrt{\frac{\overline{\theta}}{\pi}}
+O(\overline{\theta})
\qquad(\overline{\theta}\to 0^+).
\end{equation}
In particular, for every integer $q$ sufficiently large there exists $\varepsilon_q\in(0,\varepsilon_\ast)$ such that
\begin{equation}\label{eq:cdf:Theta_qminus1_over_q}
\Theta(\varepsilon_q)=\pi-\frac{\pi}{q}=\frac{q-1}{q}\,\pi.
\end{equation}
\end{corollary}

\begin{proof}
By Lemma~\ref{lem:cdf:theta-bar-expansion}, the map $\varepsilon\mapsto\overline{\theta}(\varepsilon)$ is strictly monotone on each side of $0$.
Hence the restrictions $\overline{\theta}:(0,\varepsilon_\ast)\to(0,\theta_\ast)$ and $\overline{\theta}:(-\varepsilon_\ast,0)\to(0,\theta_\ast)$
are bijections for $\varepsilon_\ast$ small, and each has a $C^\infty$ inverse by the Inverse Function Theorem
(applied at any $\varepsilon\neq 0$, using \eqref{eq:cdf:theta-bar-derivative}).
This yields the inverses $\varepsilon_\pm(\overline{\theta})$.

For the asymptotic, rewrite \eqref{eq:cdf:theta-bar-expansion} as
$\overline{\theta}=\pi\varepsilon^2(1+O(\varepsilon))$ and solve for $\varepsilon$ on each side, giving \eqref{eq:cdf:eps-asymp}.
Finally, take $\overline{\theta}=\pi/q$ and set $\varepsilon_q:=\varepsilon_+(\pi/q)$ to obtain \eqref{eq:cdf:Theta_qminus1_over_q}.
\end{proof}

\subsection{Dihedral gluing and proof of Theorem~\ref{thm:intro:cdf-epi}}
\label{subsec:cdf:dihedral-gluing}

We now explain how fundamental arcs yield closed epicyclic shrinkers for curve diffusion flow.
The construction is formally identical to the jellyfish case (Subsection~\ref{subsec:dihedral-gluing}),
so we only record the genuinely CDF-specific input: the seam condition is expressed via the {radial seam functional} $B$
and the degenerate endpoint condition $v(L)=0$ has been replaced by $\Phi=0$ (Definition~\ref{def:cdf:fundamental-arc}).

\subsubsection{Reflection principle and dihedral concatenation}

For a line $\ell$ in $\R^2$ denote by $\sigma_\ell$ the Euclidean reflection across $\ell$.

\begin{lemma}[Reflection principle for the CDF homothetic ODE]\label{lem:reflection-principle-cdf}
Let $\gamma:[0,L]\to\R^2$ be an arc-length parametrised solution of the homothetic shrinker equation
\eqref{eq:intro:homot} on $(0,L)$. Assume that $\gamma(L)\in\ell$, that $\gamma$ meets $\ell$ orthogonally at $s=L$,
and that $k_s(L)=0$. Then the reflected arc
\[
\widetilde\gamma(s):=\sigma_\ell\bigl(\gamma(2L-s)\bigr),
\qquad s\in[L,2L],
\]
glues to $\gamma$ to give a $C^\infty$ solution on $[0,2L]$. The same holds at $s=0$.
\end{lemma}

\begin{proof}
The proof is identical to Lemma~\ref{lem:reflection-principle-jellyfish}: orthogonality gives matching of the tangent,
$k_s(L)=0$ gives even continuation of curvature across the seam, and higher regularity follows by differentiating the Frenet system.
Invariance of \eqref{eq:intro:homot} under Euclidean reflections yields that the reflected arc satisfies the same equation.
\end{proof}

Fix $\overline{\theta}\in(0,\theta_\ast)$ and choose $\varepsilon=\varepsilon_+(\overline{\theta})$ from
Corollary~\ref{cor:cdf:eps-as-fn-of-theta}, together with the corresponding fundamental arc
$\gamma:[0,L]\to\R^2$ given by Proposition~\ref{prop:cdf:fundamental-arc}.
Set $c_\varepsilon:=(-1/\varepsilon,0)$ and let $\ell_0$ (resp.\ $\ell_1$) be the line through $c_\varepsilon$ and $\gamma(0)$
(resp.\ through $c_\varepsilon$ and $\gamma(L)$).
By \eqref{eq:cdf:B-is-RT} and $B(\alpha(\varepsilon),\varepsilon,L(\varepsilon))=0$, the endpoint condition $B=0$
is precisely the radial orthogonality condition $R(L)\cdot T(L)=0$, i.e.\ $\gamma$ meets $\ell_1$ orthogonally at $s=L$.
Moreover, since $\Phi=0$ and $\alpha\neq 0$, we have $v(L)=k_s(L)=0$ by \eqref{eq:cdf:vPhi}. The same conditions hold at $s=0$
by the normalisation $v(0)=0$ and the fact that $R(0)$ is radial.

Define the doubled arc $\Gamma:[0,2L]\to\R^2$ by reflecting $\gamma$ across $\ell_1$ at $s=L$:
\[
\Gamma(s):=
\begin{cases}
\gamma(s), & s\in[0,L],\\
\sigma_{\ell_1}\!\big(\gamma(2L-s)\big), & s\in[L,2L].
\end{cases}
\]
By Lemma~\ref{lem:reflection-principle-cdf}, $\Gamma$ is $C^\infty$ and satisfies \eqref{eq:intro:homot} on $[0,2L]$.

Let $\rho:=\sigma_{\ell_1}\circ\sigma_{\ell_0}$. Since $\ell_0$ and $\ell_1$ are distinct lines meeting at $c_\varepsilon$,
$\rho$ is a rotation about $c_\varepsilon$ through angle $2\Theta$, where $\Theta=\theta(L)$ is the seam angle between $\ell_0$ and $\ell_1$.
In particular, if $\Theta=\frac{p}{q}\pi$ with $\gcd(p,q)=1$, then $\rho^q=\mathrm{Id}$.

\begin{proposition}[Dihedral closure from a fundamental arc]\label{prop:cdf:dihedral-closure}
Assume the seam angle satisfies
\[
\frac{\Theta}{\pi}\in\Q,
\qquad\text{equivalently}\qquad
\Theta=\frac{p}{q}\,\pi \ \text{ for coprime integers }p,q\ge 1.
\]
Then the concatenation
\[
\gamma_{p/q}:=\Gamma_0\#\Gamma_1\#\cdots\#\Gamma_{q-1},
\qquad
\Gamma_j:=\rho^j\circ\Gamma,
\]
defines a $C^\infty$ closed immersed curve $\gamma_{p/q}:\S\to\R^2$ satisfying \eqref{eq:intro:homot}.
Moreover, $\gamma_{p/q}$ is invariant under the dihedral group generated by $\rho$ and $\sigma_{\ell_0}$, hence has dihedral symmetry
of order $q$.
\end{proposition}

\begin{proof}
This is identical to the proof of Proposition~\ref{prop:dihedral-closure} in the jellyfish section:
each $\Gamma_j$ is a rigid motion image of $\Gamma$ and hence satisfies \eqref{eq:intro:homot};
seam matching follows from the reflection principle and the fact that $\rho$ maps $\ell_0$ to $\ell_1$;
and $\rho^q=\mathrm{Id}$ yields closure.
\end{proof}

\begin{proof}[Proof of Theorem~\ref{thm:intro:cdf-epi}]
Choose $q_0$ so large that $\pi/q_0<\theta_\ast$ (from Corollary~\ref{cor:cdf:eps-as-fn-of-theta}).
For each integer $q\ge q_0$, set $\overline{\theta}=\pi/q$ and obtain $\varepsilon_q$ with
$\Theta(\varepsilon_q)=\frac{q-1}{q}\pi$ by \eqref{eq:cdf:Theta_qminus1_over_q}.
Then Proposition~\ref{prop:cdf:dihedral-closure} produces a smooth closed curve $\gamma_{(q-1)/q}$ with dihedral symmetry of order $q$
satisfying the homothetic shrinker equation \eqref{eq:intro:homot}, hence a self-similarly shrinking solution of curve diffusion flow.
Geometric distinctness is proved in Subsection~\ref{subsec:cdf:distinctness} below.
\end{proof}

\subsection{Maximality of the symmetry order and geometric distinctness}
\label{subsec:cdf:distinctness}

As in the jellyfish case, geometric distinctness follows by showing that there are no additional ``hidden'' reflection axes.
For curve diffusion flow, the decisive obstruction is the absence of {interior radial seams} on the fundamental arc.

\begin{lemma}[No interior radial seams on a fundamental arc]\label{lem:cdf:no-interior-seams}
Let $|\varepsilon|$ be sufficiently small and let $\gamma:[0,L(\varepsilon)]\to\R^2$ be the fundamental arc from
Proposition~\ref{prop:cdf:fundamental-arc}. Then the radial seam function
\[
s\longmapsto
-\sin\theta(s)+\varepsilon\bigl(y(s)\cos\theta(s)-x(s)\sin\theta(s)\bigr)
\]
vanishes only at $s=0$ and $s=L(\varepsilon)$. Equivalently, $R(s)\cdot T(s)\neq 0$ for all $s\in(0,L(\varepsilon))$.
\end{lemma}

\begin{proof}
For the base semicircle \eqref{eq:cdf:circle} at $(\alpha,\varepsilon)=(0,0)$, the seam function equals $-\sin s$,
which has exactly two simple zeros on $[0,\pi]$ (at $s=0$ and $s=\pi$) and is strictly negative on $(0,\pi)$.

By smooth dependence of solutions on parameters (as in Lemma~\ref{lem:smooth-dependence} of the jellyfish section),
the seam function depends continuously on $(s,\alpha,\varepsilon)$ on compact subsets.
Hence, for $|\varepsilon|$ sufficiently small, the two simple zeros persist near $0$ and near $L(\varepsilon)$,
and no new zero can appear unless a multiple root forms, i.e.\ the seam function and its $s$-derivative vanish simultaneously.
Since $-\sin s$ has no multiple root on $[0,\pi]$, such a bifurcation is excluded for small perturbations.
Therefore there are no interior zeros.
\end{proof}

\begin{lemma}[No extra reflection axes]\label{lem:cdf:no-extra-axes}
Let $\gamma_{p/q}$ be a closed curve obtained by the dihedral gluing in Proposition~\ref{prop:cdf:dihedral-closure}.
Then the full symmetry group of $\gamma_{p/q}$ is exactly the dihedral group $D_q$; in particular,
$\gamma_{p/q}$ does {not} admit dihedral symmetry of order $\hat q$ for any $\hat q>q$.
\end{lemma}

\begin{proof}
The curve is invariant under $D_q$ by construction.
If it had dihedral symmetry of order $\hat q>q$ about the same centre $c_\varepsilon$, then there would exist a reflection axis
$\ell$ strictly between two adjacent seam rays used in the construction.
At any fixed point of a reflection, the curve meets the axis orthogonally; in centre coordinates this is exactly the radial seam
condition $R\cdot T=0$.
Such an interior seam would occur on the fundamental arc portion between adjacent seam rays, contradicting
Lemma~\ref{lem:cdf:no-interior-seams}.
\end{proof}

For epicyclic shrinkers it is convenient to use two similarity invariants: the dihedral symmetry order $q$ and the
rotation index (turning number), determined by the rational parameter $\Theta=\frac{p}{q}\pi$.

\begin{lemma}[Rotation index]\label{lem:cdf:turning-number}
Let $\gamma_{p/q}$ be the closed curve from Proposition~\ref{prop:cdf:dihedral-closure} with $\Theta=\frac{p}{q}\pi$.
Then the total turning of the tangent is $2q\Theta=2p\pi$. In particular, the rotation index of $\gamma_{p/q}$ equals $p$.
\end{lemma}

\begin{proof}
Along the fundamental arc the tangent angle increases by $\Theta$, doubling increases it by $2\Theta$, and concatenating $q$ rotated copies
increases it by $2q\Theta=2p\pi$. The rotation index is total turning divided by $2\pi$.
\end{proof}

\begin{corollary}[Non-equivalence of epicyclic shrinkers]\label{cor:cdf:non-equivalence}
Let $\gamma_{p/q}$ and $\gamma_{\hat p/\hat q}$ be closed epicyclic shrinkers produced by Proposition~\ref{prop:cdf:dihedral-closure}
with coprime pairs $(p,q)$ and $(\hat p,\hat q)$. If $(p,q)\neq(\hat p,\hat q)$, then $\gamma_{p/q}$ is not equivalent to
$\gamma_{\hat p/\hat q}$ under any similarity transformation.
\end{corollary}

\begin{proof}
Similarities preserve the rotation index and conjugate the full symmetry group.
If $p\neq \hat p$, then Lemma~\ref{lem:cdf:turning-number} rules out similarity equivalence.
If $p=\hat p$ but $q\neq \hat q$, then Lemma~\ref{lem:cdf:no-extra-axes} implies the full symmetry groups are $D_q$ and $D_{\hat q}$,
which have different orders $2q\neq 2\hat q$ and hence cannot be conjugate. Therefore no similarity can map one curve to the other.
\end{proof}

\subsection{An abundance of angles}

In order for us to produce a smooth closed curve, we require the total change in angle (which we call the seam angle, see below) along a fundamental arc to be non-trivial.

\begin{lemma}\label{lem:theta-bar-expansion}
Let $\alpha(\varepsilon)$ and $L(\varepsilon)$ be the functions provided by Proposition~\ref{prop:fundamental-arc}, so that
\[
B\bigl(L(\varepsilon),\alpha(\varepsilon),\varepsilon\bigr)=(0,0),
\qquad \alpha(0)=0,\quad L(0)=L_0.
\]
Define the \emph{seam angle}
\begin{equation}\label{eq:theta-bar-def}
\overline{\theta}(\varepsilon)
:=
\theta\bigl(L(\varepsilon);\alpha(\varepsilon),\varepsilon\bigr).
\end{equation}
Then $\overline{\theta}$ is $C^\infty$ with $\overline{\theta}(0)=0$, and
\begin{equation}\label{eq:theta-bar-firstorder}
\overline{\theta}(\varepsilon)=\overline{\theta}'(0)\,\varepsilon+O(\varepsilon^2)
\qquad(\varepsilon\to 0),
\end{equation}
where
\begin{equation}\label{eq:theta-bar-derivative}
\overline{\theta}'(0)=y_0(L_0)>0.
\end{equation}
In particular, $\varepsilon\mapsto \overline{\theta}(\varepsilon)$ is strictly monotone for $|\varepsilon|$ sufficiently small.
\end{lemma}

\begin{proof}
Smoothness follows from Lemma~\ref{lem:smooth-dependence} and Proposition~\ref{prop:fundamental-arc}.
We compute the derivative using the chain rule.  Write all partial derivatives at the base point
\[
(s,\alpha,\varepsilon)=(L_0,0,0),
\qquad\text{so that}\qquad
\theta_0(L_0)=0,\ \ k_0(L_0)=-1,\ \ v_0(L_0)=0.
\]
Differentiating \eqref{eq:theta-bar-def} gives
\begin{equation}\label{eq:theta-bar-chain}
\overline{\theta}'(0)
=
\theta_s(L_0;0,0)\,L'(0)+\theta_\alpha(L_0;0,0)\,\alpha'(0)+\theta_\varepsilon(L_0;0,0).
\end{equation}
Since $\theta_s=k$, we have $\theta_s(L_0;0,0)=k_0(L_0)=-1$.

Next, $\theta_\varepsilon(L_0;0,0)=0$: indeed, at $\alpha=0$ the curvature subsystem coincides with \eqref{eq:base-curv} and
is independent of $\varepsilon$, hence $k(\,\cdot\,;0,\varepsilon)\equiv k_0$ and therefore
$\theta(\,\cdot\,;0,\varepsilon)\equiv \theta_0$ for all $\varepsilon$.

It remains to compute $L'(0)$ and eliminate the $\alpha'(0)$ term.  Recall that $L(\alpha,\varepsilon)$ is defined by
$b_1(L(\alpha,\varepsilon),\alpha,\varepsilon)=0$ (Proposition~\ref{prop:L-of-alpha-eps}) and that
$L(\varepsilon)=L(\alpha(\varepsilon),\varepsilon)$. Thus
\begin{equation}\label{eq:Lprime}
L'(0)=\partial_\alpha L(0,0)\,\alpha'(0)+\partial_\varepsilon L(0,0).
\end{equation}
Differentiate the identity $b_1(L(\alpha,\varepsilon),\alpha,\varepsilon)\equiv 0$ in $\alpha$ at $(0,0)$.
At $\varepsilon=0$ one has $b_1=-\sin\theta$, hence
\[
0=\partial_s b_1(L_0,0,0)\,\partial_\alpha L(0,0)+\partial_\alpha b_1(L_0,0,0)
= \partial_s b_1(L_0,0,0)\,\partial_\alpha L(0,0)-\theta_\alpha(L_0;0,0).
\]
Using $\partial_s b_1(L_0,0,0)=1$ from \eqref{eq:b1s-base} yields
\begin{equation}\label{eq:Lalpha-equals-thetaalpha}
\partial_\alpha L(0,0)=\theta_\alpha(L_0;0,0).
\end{equation}
Substituting \eqref{eq:Lprime} and \eqref{eq:Lalpha-equals-thetaalpha} into \eqref{eq:theta-bar-chain}, and using
$\theta_s(L_0;0,0)=-1$ and $\theta_\varepsilon(L_0;0,0)=0$, we obtain the cancellation
\[
\overline{\theta}'(0)
=
(-1)\Bigl(\theta_\alpha(L_0;0,0)\alpha'(0)+\partial_\varepsilon L(0,0)\Bigr)
+\theta_\alpha(L_0;0,0)\alpha'(0)
=
-\partial_\varepsilon L(0,0).
\]
Finally, compute $\partial_\varepsilon L(0,0)$ from $b_1(L(0,\varepsilon),0,\varepsilon)\equiv 0$.
Differentiating in $\varepsilon$ at $\varepsilon=0$ gives
\[
0=\partial_s b_1(L_0,0,0)\,\partial_\varepsilon L(0,0)+\partial_\varepsilon b_1(L_0,0,0).
\]
Again $\partial_s b_1(L_0,0,0)=1$, and since $\theta_0(L_0)=0$ we have
$\partial_\varepsilon b_1(L_0,0,0)=y_0(L_0)\cos\theta_0(L_0)-x_0(L_0)\sin\theta_0(L_0)=y_0(L_0)$.
Hence $\partial_\varepsilon L(0,0)=-y_0(L_0)$, and therefore \eqref{eq:theta-bar-derivative} holds.

Positivity follows from $y_0'(s)=\cos\theta_0(s)=k_0(s)^2\ge 0$ and $k_0\not\equiv 0$ on $(0,L_0)$, so
$y_0(L_0)=\int_0^{L_0}k_0(s)^2\,ds>0$.
The first-order expansion \eqref{eq:theta-bar-firstorder} follows.
\end{proof}

We thus conclude that our fundamental arcs turn through an abundance of seam angles.

\begin{corollary}\label{cor:eps-as-fn-of-theta}
There exist $\theta_\ast>0$,  $\varepsilon_\ast>0$ and a unique $C^\infty$ function
\[
\varepsilon=\varepsilon(\overline{\theta})\qquad\text{for }|\overline{\theta}|<\theta_\ast,
\]
with $\varepsilon(0)=0$, such that
\[
\theta\bigl(L(\varepsilon(\overline{\theta})),\alpha(\varepsilon(\overline{\theta})),\varepsilon(\overline{\theta})\bigr)
=\overline{\theta}.
\]
Moreover,
\begin{equation}\label{eq:eps-asymp}
\varepsilon(\overline{\theta})
=
\frac{1}{y_0(L_0)}\,\overline{\theta}+O(\overline{\theta}^2)
\qquad(\overline{\theta}\to 0).
\end{equation}
\end{corollary}
\begin{proof}
By Lemma~\ref{lem:theta-bar-expansion}, $\overline{\theta}'(0)=y_0(L_0)\neq 0$, so the Inverse Function Theorem
applies to the map $\varepsilon\mapsto\overline{\theta}(\varepsilon)$ and yields a local $C^\infty$ inverse
$\varepsilon=\varepsilon(\overline{\theta})$ defined for $|\overline{\theta}|$ small. The expansion \eqref{eq:eps-asymp}
follows by inverting the first-order Taylor expansion \eqref{eq:theta-bar-firstorder}.
The final statement follows by taking $\overline{\theta}=\pm\pi/m$ with $m$ large enough.
\end{proof}

In particular, for all integers $m$ sufficiently large (so that $\pi/m<\theta_\ast$) there exists a unique
$\varepsilon_m$ with $0 < \varepsilon_m < \varepsilon_\ast$ such that the corresponding fundamental arc satisfies
$\overline{\theta}(\varepsilon_m)=\pi/m$.
This yields the seam angle required for the dihedral gluing in
Proposition~\ref{prop:dihedral-closure}.

\subsection{Dihedral gluing and proof of Theorem~\ref{thm:intro:jellyfish}}
\label{subsec:dihedral-gluing}

We now explain how fundamental arcs yield jellyfish.
We refer the reader to Figure \ref{EFpic} for a visual representation, and for a different example, the leftmost plot in Figure \ref{fig:intro:representatives}.

\begin{figure}[t]
\centering

\begin{minipage}[t]{0.29\textwidth}
  \vspace{0pt}\centering 

  \begin{subfigure}[t]{\linewidth}
    \centering
  \incfig{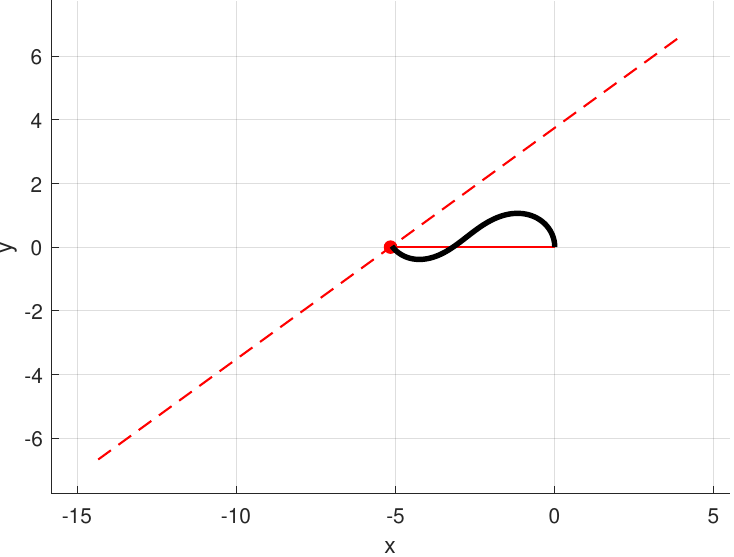}{\linewidth}{8mm}
  \caption{Fundamental arc.}
  \label{fig:fef:1_5:fundamental}
  \end{subfigure}

  \vspace{2mm}

  \begin{subfigure}[t]{\linewidth}
    \centering
  \incfig{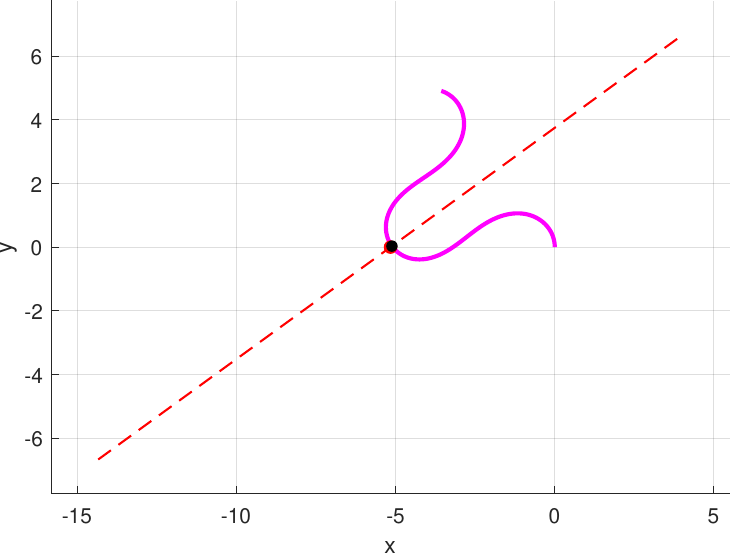}{\linewidth}{8mm}
  \caption{Doubled arc.}
  \label{fig:fef:1_5:doubled}
  \end{subfigure}
\end{minipage}\hfill
\begin{minipage}[t]{0.67\textwidth}
  \vspace{0pt}\centering 

  \begin{subfigure}[t]{\linewidth}
    \centering
\incfigl{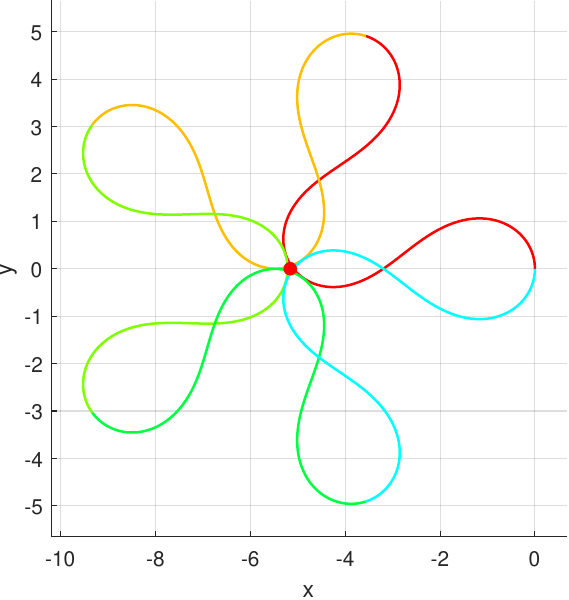}{0.75\linewidth}{0}{0}
\label{fig:fef:1_7}
  \end{subfigure}
\end{minipage}
\caption{Free elastic flow jellyfish solution with $\omega=1$ and $\varepsilon=0.19378$ (to five significant figures).
The closed curve is assembled from $5$ doubled arcs.}
\label{EFpic}
\end{figure}

\subsubsection{Reflection principle and dihedral concatenation}

For a line $\ell$ in $\R^2$ denote by $\sigma_\ell$ the Euclidean reflection across $\ell$.
We shall use the following reflection principle for solutions of the homothetic ODE system; it is a
direct consequence of the even/odd continuation of the Frenet data and the invariance of \eqref{eq:intro:homot}
under Euclidean isometries.

\begin{lemma}\label{lem:reflection-principle-jellyfish}
Let $\gamma:[0,L]\to\R^2$ be an arc-length parametrised solution of the homothetic expander equation \eqref{eq:intro:homot}
(for the elastic flow) on $(0,L)$.
Assume that $\gamma(L)\in\ell$ and $\gamma$ meets $\ell$ orthogonally at $s=L$, and additionally $k_s(L)=0$.
Then the reflected arc $\widetilde\gamma(s):=\sigma_\ell(\gamma(2L-s))$ for $s\in[L,2L]$ glues to $\gamma$ to give a
$C^\infty$ solution on $[0,2L]$. The same holds at $s=0$.
\end{lemma}

\begin{proof}
Orthogonality implies the tangent reflects with matching first derivative, and the condition $k_s(L)=0$ gives the
evenness of curvature across the seam; higher regularity follows by differentiating the Frenet system and using the
smoothness of the ODE. Invariance of \eqref{eq:intro:homot} under $\sigma_\ell$ yields that the reflected arc also satisfies
the same equation.
\end{proof}

Fix $\overline{\theta}\in(0,\overline{\theta}_\ast)$ and choose $\varepsilon=\varepsilon(\overline{\theta})$ from Corollary~\ref{cor:eps-as-fn-of-theta}, together with the
corresponding fundamental arc $\gamma:[0,L]\to\R^2$ given by Proposition~\ref{prop:fundamental-arc}.
Set $c_\varepsilon:=(-1/\varepsilon,0)$. Let $\ell_0$ (resp.\ $\ell_1$) be the line through $c_\varepsilon$ and $\gamma(0)$
(resp.\ through $c_\varepsilon$ and $\gamma(L)$). By Remark~\ref{rem:b1-geometry} and $v(0)=v(L)=0$, the hypotheses of
Lemma~\ref{lem:reflection-principle-jellyfish} hold at both endpoints with $\ell=\ell_0$ and $\ell=\ell_1$.

Define the length-$2L$ block $\Gamma:[0,2L]\to\R^2$ by reflecting $\gamma$ across $\ell_1$ at $s=L$:
\[
\Gamma(s):=
\begin{cases}
\gamma(s), & s\in[0,L],\\
\sigma_{\ell_1}\!\big(\gamma(2L-s)\big), & s\in[L,2L].
\end{cases}
\]
By Lemma~\ref{lem:reflection-principle-jellyfish}, $\Gamma$ is $C^\infty$ and satisfies \eqref{eq:intro:homot} on $[0,2L]$.

Let $\rho:=\sigma_{\ell_1}\circ\sigma_{\ell_0}$. Since $\ell_0$ and $\ell_1$ are distinct lines meeting at $c_\varepsilon$,
$\rho$ is a rotation about $c_\varepsilon$ through angle $2\overline{\theta}$, where $\overline{\theta}$ is the (acute) angle between $\ell_0$ and $\ell_1$.
In particular, if $\overline{\theta}=\pi/m$ then $\rho^m=\mathrm{Id}$.

\begin{proposition}[Dihedral closure from a fundamental arc]\label{prop:dihedral-closure}
Assume $\overline{\theta}=\pi/m$ for some integer $m\ge 2$. Then the concatenation
\[
\gamma_m^j := \Gamma_0\#\Gamma_1\#\cdots\#\Gamma_{m-1},
\qquad
\Gamma_j:=\rho^j\circ\Gamma,
\]
defines a $C^\infty$ closed immersed curve $\gamma_m^j:\S\to\R^2$ satisfying \eqref{eq:intro:homot} for the elastic flow.
Moreover, $\gamma_m^j$ is invariant under the dihedral group generated by $\rho$ and $\sigma_{\ell_0}$, hence has dihedral symmetry
of order $m$.
\end{proposition}

\begin{proof}
Each $\Gamma_j$ is $C^\infty$ and satisfies \eqref{eq:intro:homot} because \eqref{eq:intro:homot} is invariant under
Euclidean isometries. The seam matching at the joins follows from the reflection principle at $s=0$ for $\Gamma$ and the fact that
$\rho$ maps $\ell_0$ to $\ell_1$ and preserves orthogonality. Since $\rho^m=\mathrm{Id}$ when $\overline{\theta}=\pi/m$, the concatenation closes.
Dihedral invariance is immediate from the construction.
\end{proof}

\subsubsection{Maximality of the symmetry order and non-equivalence for different $m$}

For our jellyfish curves geometric distinctness follows
from a monotonicity property of the curvature on the fundamental arc.

\begin{lemma}[Monotone curvature on the fundamental arc]\label{lem:v-negative}
Let $\varepsilon>0$ be sufficiently small and let $S(\,\cdot\,;\alpha(\varepsilon),\varepsilon)$ be the fundamental arc given by
Proposition~\ref{prop:fundamental-arc}, defined on $[0,L(\varepsilon)]$.
Then
\[
v(s;\alpha(\varepsilon),\varepsilon)<0\qquad\text{for all }s\in(0,L(\varepsilon)).
\]
Equivalently, $k(\,\cdot\,;\alpha(\varepsilon),\varepsilon)$ is strictly decreasing on $[0,L(\varepsilon)]$.
\end{lemma}

\begin{proof}
For the base arc at $(\alpha,\varepsilon)=(0,0)$ we have $v_0(s)<0$ for all $s\in(0,L_0)$.
Fix $\tau\in(0,L_0/2)$. By continuity and compactness there exists $\eta_\tau>0$ such that
\[
v_0(s)\le -\eta_\tau\qquad\text{for all } s\in[\tau,L_0-\tau].
\]
By smooth dependence of solutions on parameters (Lemma~\ref{lem:smooth-dependence}), the map
$(s,\alpha,\varepsilon)\mapsto v(s;\alpha,\varepsilon)$ is continuous on a neighbourhood of
$[\tau,L_0-\tau]\times\{(0,0)\}$, hence there is $\delta_\tau>0$ such that
\[
v(s;\alpha,\varepsilon)\le -\frac{\eta_\tau}{2}
\qquad\text{for all }s\in[\tau,L_0-\tau],\ \ |\alpha|+|\varepsilon|<\delta_\tau.
\]
Now choose $\varepsilon$ small so that $(\alpha(\varepsilon),\varepsilon)$ lies in this neighbourhood and
$|L(\varepsilon)-L_0|<\tau$ (which holds for $\varepsilon$ sufficiently small since $L(\varepsilon)\to L_0$).
Then we already have $v<0$ on $[\tau,L(\varepsilon)-\tau]$.

It remains to control the two short end intervals $(0,\tau)$ and $(L(\varepsilon)-\tau,L(\varepsilon))$.
Since $v_0'(0)<0$, there exists $\tau>0$ such that $v_0(s)<0$ for all $s\in(0,\tau]$; by continuity in parameters,
after shrinking $\tau$ and $\delta_\tau$ if necessary we obtain $v(s;\alpha(\varepsilon),\varepsilon)<0$ for all
$s\in(0,\tau]$. An identical argument near $s=L_0$ using $v_0'(L_0)>0$ yields negativity on
$[L(\varepsilon)-\tau,L(\varepsilon))$ as well (recall $v(L(\varepsilon);\alpha(\varepsilon),\varepsilon)=0$ by construction).
Combining these intervals proves $v<0$ on $(0,L(\varepsilon))$.
\end{proof}

The key idea is that there can not be any additional `hidden' axes of reflection symmetry for fundamental arcs.
Their constructed symmetry order is maximal.

\begin{lemma}\label{lem:no-extra-axes}
Let $\gamma_m^j$ be a closed curve obtained from a fundamental arc by the dihedral gluing procedure in
Proposition~\ref{prop:dihedral-closure}, with prescribed seam angle $\overline\theta=\pi/m$.
Then the full symmetry group of $\gamma_m^j$ is exactly the dihedral group $D_m$; in particular,
$\gamma_m^j$ does {not} admit dihedral symmetry of order $\hat m$ for any $\hat m>m$.
\end{lemma}

\begin{proof}
By construction, $\gamma_m^j$ is invariant under the dihedral group $D_m$ generated by reflections across two adjacent
seam lines through the homothety centre $c$ meeting at angle $\pi/m$. Thus the full symmetry group contains $D_m$.

Suppose for contradiction that $\gamma_m^j$ has dihedral symmetry of order $\hat m>m$ about the same centre $c$.
Then the set of reflection axes of $D_{\hat m}$ consists of $2\hat m$ distinct lines through $c$, spaced by angles $\pi/\hat m$.
Since $\hat m>m$, there exists at least one such reflection axis $\ell$ strictly between the two {adjacent} seam lines
used in the construction (their angle is $\pi/m$), i.e.\ $\ell$ meets the interior of the corresponding fundamental wedge.

Let $\Gamma$ denote the piece of $\gamma_m^j$ lying in this fundamental wedge between the adjacent seam lines.
The reflection symmetry across $\ell$ fixes the set $\Gamma\cap\ell$. Choose a point $p\in \Gamma\cap\ell$ lying strictly inside
the wedge (such a point exists because $\ell$ passes through the wedge interior and $\Gamma$ connects its boundary lines).

Parametrise $\Gamma$ by arclength so that $p=\Gamma(s_\ast)$ with $s_\ast\in(0,L)$, where $L$ is the length of the fundamental arc.
Reflection across $\ell$ reverses the orientation along $\Gamma$ while fixing $p$, hence the curvature function
is even about $s_\ast$:
\[
k(s_\ast+t)=k(s_\ast-t)\qquad\text{for all $t$ sufficiently small.}
\]
Therefore $k_s(s_\ast)=0$, i.e.\ $v(s_\ast)=0$.

This contradicts Lemma~\ref{lem:v-negative}, which asserts that on each fundamental arc segment between adjacent seam points we have
$v<0$ in the interior. Hence no such additional axis exists and the full symmetry group is precisely $D_m$.
\end{proof}

\begin{corollary}[Non-equivalence for different symmetry orders]\label{cor:non-equivalence-m}
If $m\neq \hat m$, then $\gamma_m^j$ is not equivalent to $\gamma_{\hat m}^j$ under any similarity transformation.
\end{corollary}

\begin{proof}
A similarity $\Phi$ conjugates the full symmetry group: $\mathrm{Sym}(\Phi(\gamma))=\Phi\,\mathrm{Sym}(\gamma)\,\Phi^{-1}$.
In particular, the abstract group (hence its order) is preserved.
By Lemma~\ref{lem:no-extra-axes}, the full symmetry group of $\gamma_m^j$ is exactly $D_m$, while that of $\gamma_{\hat m}^j$
is exactly $D_{\hat m}$. Since $|D_m|=2m\neq 2\hat m=|D_{\hat m}|$, these groups are non-isomorphic, hence cannot be conjugate.
Therefore no similarity can map $\gamma_m^j$ to $\gamma_{\hat m}^j$.
\end{proof}

\subsubsection{Proof of Theorem~\ref{thm:intro:jellyfish}}

\begin{proof}[Proof of Theorem~\ref{thm:intro:jellyfish}]
By Corollary~\ref{cor:eps-as-fn-of-theta} there exists $\overline{\theta}_\ast>0$ such that for every $\overline{\theta}\in(0,\overline{\theta}_\ast)$
there is a fundamental arc for \eqref{eq:ode-S} with that value of $\overline{\theta}$.
Choose $m_0\in\N$ such that $\pi/m_0<\overline{\theta}_\ast$. Then for every integer $m>m_0$ we may set $\overline{\theta}=\pi/m\in(0,\overline{\theta}_\ast)$
and obtain a corresponding fundamental arc. Proposition~\ref{prop:dihedral-closure} then produces a smooth closed curve
$\gamma_m^j$ with dihedral symmetry of order $m$ satisfying the homothetic expander equation \eqref{eq:intro:homot}
for the elastic flow, hence a self-similarly expanding solution of \eqref{EF}.
This proves existence for all $m>m_0$.
Finally, Corollary~\ref{cor:non-equivalence-m} shows that $\gamma_m^j$ and $\gamma_{\hat m}^j$ are not equivalent under similarity
whenever $m\neq \hat m$, as claimed.
\end{proof}

\section{Epicyclic shrinkers for curve diffusion flow}
\label{sec:cdf-epicycles}

In this section we prove Theorem~\ref{thm:intro:cdf-epi}.  Our strategy is similar to that of the jellyfish construction
(Section~\ref{sec:jellyfish-expanders}), with some additional complications.
One essential difference is that for curve diffusion flow the endpoint condition
$k_s(L)=0$ is {degenerate} along the base family $\alpha=0$, and must be replaced by an equivalent smooth
condition after dividing out the trivial factor $\alpha$.
A second technical difficulty comes in showing that the fundamental arcs turn through a non-trivial angle.

\subsection{The fundamental arc}
\label{subsec:cdf-fundamental-arc}

\subsubsection{The ODE and boundary data}

Fix parameters $(\alpha,\varepsilon)$ near $(0,0)$.  We consider the state vector
\[
S(s;\alpha,\varepsilon):=(x(s),y(s),\theta(s),k(s),v(s))\in\R^5
\]
solving
\begin{equation}\label{eq:cdf:ode-S}
\left\{
\begin{aligned}
x'&=-\sin\theta,\\
y'&=\phantom{-}\cos\theta,\\
\theta'&=k,\\
k'&=v,\\
v'&=\alpha\cos\theta+\alpha\varepsilon\bigl(x\cos\theta+y\sin\theta\bigr),
\end{aligned}
\right.
\qquad
S(0;\alpha,\varepsilon)=(1,0,0,1,0).
\end{equation}
As in the jellyfish section we introduce the seam functional
\begin{equation}\label{eq:cdf:Bdef}
B(\alpha,\varepsilon,L)
:=
-\sin\theta(L)+\varepsilon\bigl(y(L)\cos\theta(L)-x(L)\sin\theta(L)\bigr),
\end{equation}
so that the endpoint seam condition is $B(\alpha,\varepsilon,L)=0$.

The second endpoint condition is $v(L)=0$ (equivalently $k_s(L)=0$).  However, unlike the free elastic flow case,
this condition is {automatically satisfied} whenever $\alpha=0$, hence cannot be used directly as a nondegenerate
shooting equation.  Indeed integrating the last equation of \eqref{eq:cdf:ode-S} gives the factorisation
\begin{equation}\label{eq:cdf:v-factor}
v(L;\alpha,\varepsilon)
=
\alpha\int_0^L\Big(\cos\theta+\varepsilon(x\cos\theta+y\sin\theta)\Big)\,ds.
\end{equation}
We define
\begin{equation}\label{eq:cdf:PhiDef}
\Phi(\alpha,\varepsilon,L)
:=
-\int_0^L\Big(\cos\theta+\varepsilon(x\cos\theta+y\sin\theta)\Big)\,ds,
\end{equation}
so that \eqref{eq:cdf:v-factor} becomes
\begin{equation}\label{eq:cdf:vPhi}
v(L;\alpha,\varepsilon)= -\alpha\,\Phi(\alpha,\varepsilon,L).
\end{equation}
Thus, any solution with $\alpha\neq 0$ satisfies $\Phi(\alpha,\varepsilon,L)=0$ and thus may not be circular.
We make the following definition.

\begin{definition}\label{def:cdf:fundamental-arc}
A \emph{fundamental arc} is a solution of \eqref{eq:cdf:ode-S} for which there exists $L>0$ and $\alpha\neq 0$ such that
\begin{equation}\label{eq:cdf:fundamental-conditions}
\Phi(\alpha,\varepsilon,L)=0
\qquad\text{and}\qquad
B(\alpha,\varepsilon,L)=0.
\end{equation}
Equivalently, $v(L)=0$ and $B(\alpha,\varepsilon,L)=0$ with $\alpha\neq 0$.
\end{definition}

\paragraph{Smooth dependence.}
As in Section~\ref{sec:jellyfish-expanders}, the right-hand side of \eqref{eq:cdf:ode-S} is smooth in $(S,\alpha,\varepsilon)$, hence
$(s,\alpha,\varepsilon)\mapsto S(s;\alpha,\varepsilon)$ is $C^\infty$ on a uniform interval for parameters near $(0,0)$.
Consequently, $(\alpha,\varepsilon,L)\mapsto(\Phi(\alpha,\varepsilon,L),B(\alpha,\varepsilon,L))$ is $C^\infty$.

\subsubsection{The base arc: the semicircle}

\begin{lemma}[Base semicircle]\label{lem:cdf:base-semicircle}
At $\alpha=0$ (for any $\varepsilon$) the unique solution of \eqref{eq:cdf:ode-S} is the unit circle parametrisation
\begin{equation}\label{eq:cdf:circle}
x(s)=\cos s,\qquad y(s)=\sin s,\qquad \theta(s)=s,\qquad k(s)=1,\qquad v(s)=0.
\end{equation}
Moreover, setting $L_0:=\pi$ one has $B(0,0,L_0)=0$ and $\Phi(0,0,L_0)=0$.
\end{lemma}

\begin{proof}
If $\alpha=0$ then $v'\equiv 0$, hence $v\equiv 0$ by the initial condition.  Thus $k\equiv 1$ and $\theta(s)=s$, and integrating
$x'=-\sin s$, $y'=\cos s$ with $(x(0),y(0))=(1,0)$ yields \eqref{eq:cdf:circle}.
At $L_0=\pi$ we have $\sin\theta(\pi)=0$ and $y(\pi)\cos\theta(\pi)-x(\pi)\sin\theta(\pi)=0$, hence $B(0,0,\pi)=0$.
Also \eqref{eq:cdf:PhiDef} reduces to $\Phi(0,\varepsilon,L)=-\sin L-\varepsilon L$, so $\Phi(0,0,\pi)=0$.
\end{proof}

\subsubsection{Implicit function theorem for $(\alpha,L)$ as functions of $\varepsilon$}

Set
\begin{equation}\label{eq:cdf:Fdef}
F(\alpha,\varepsilon,L)
:=
\begin{bmatrix}
\Phi(\alpha,\varepsilon,L)\\
B(\alpha,\varepsilon,L)
\end{bmatrix}.
\end{equation}
Then $F(0,0,\pi)=(0,0)$ by Lemma~\ref{lem:cdf:base-semicircle}.
To facilitate application of the Implicit Function Theorem, we require a non-degeneracy condition, for which we require calculation of several derivatives.

\begin{lemma}\label{lem:cdf:jacobian}
At $(\alpha,\varepsilon,L)=(0,0,\pi)$ one has
\[
\Phi_\alpha=\frac{\pi}{2},\qquad \Phi_L=1,\qquad \Phi_\varepsilon=-\pi,
\qquad
B_\alpha=\pi,\qquad B_L=1,\qquad B_\varepsilon=0,
\]
and therefore
\begin{equation}\label{eq:cdf:det}
D_{(\alpha,L)}F(0,0,\pi)
=
\begin{bmatrix}
\frac{\pi}{2} & 1\\[1mm]
\pi & 1
\end{bmatrix},
\qquad
\det D_{(\alpha,L)}F(0,0,\pi)=-\frac{\pi}{2}\neq 0.
\end{equation}
\end{lemma}

\begin{proof}
We work at the base point $(\alpha,\varepsilon)=(0,0)$ and use the base solution from
Lemma~\ref{lem:cdf:base-semicircle}
Let us begin by differentiating $\Phi$ with respect to $L$ and $\varepsilon$.
By definition \eqref{eq:cdf:PhiDef},
 at $\alpha=0$ the solution is exactly the base circle for all $\varepsilon$, hence $\theta(s)=s$ and
\[
x\cos\theta+y\sin\theta=\cos s\cos s+\sin s\sin s=1.
\]
Therefore for all $(\varepsilon,L)$,
\begin{equation}\label{eq:cdf:Phi-alpha0}
\Phi(0,\varepsilon,L)= -\int_0^L(\cos s+\varepsilon)\,ds=-\sin L-\varepsilon L.
\end{equation}
Differentiating \eqref{eq:cdf:Phi-alpha0} yields
\[
\Phi_L(0,0,\pi)=-\cos\pi=1,
\qquad
\Phi_\varepsilon(0,0,\pi)=-\pi.
\]
Now we differentiate $B$ with respect to $L$ and $\varepsilon$. 
From \eqref{eq:cdf:Bdef} at $\varepsilon=0$ we have $B(\alpha,0,L)=-\sin\theta(L)$, hence by the chain rule
\[
B_L(\alpha,0,L)= -\cos\theta(L)\,\partial_L\theta(L).
\]
Since $\partial_L\theta(L)=\theta_s(L)=k(L)$, evaluating at the base point gives
\[
B_L(0,0,\pi)= -\cos\theta_0(\pi)\,k_0(\pi)= -\cos\pi\cdot 1=1.
\]
Next, differentiate $B$ with respect to $\varepsilon$ and then evaluate at $(0,0,\pi)$.  The explicit $\varepsilon$-dependence is linear,
and at the base circle one has
\[
y_0(\pi)\cos\theta_0(\pi)-x_0(\pi)\sin\theta_0(\pi)=\sin\pi\cos\pi-\cos\pi\sin\pi=0.
\]
Moreover, $\theta_\varepsilon(0,0,\cdot)\equiv 0$ because $\varepsilon$ enters \eqref{eq:cdf:ode-S} only multiplied by $\alpha$ and the base has $\alpha=0$.
Therefore
\[
B_\varepsilon(0,0,\pi)=0.
\]

Finally we calculate the $\alpha$-variation at $(\alpha,\varepsilon)=(0,0)$. 
Set
\[
x_1:=\partial_\alpha x\big|_{(0,0)},\quad
y_1:=\partial_\alpha y\big|_{(0,0)},\quad
\theta_1:=\partial_\alpha\theta\big|_{(0,0)},\quad
k_1:=\partial_\alpha k\big|_{(0,0)},\quad
v_1:=\partial_\alpha v\big|_{(0,0)}.
\]
Differentiating \eqref{eq:cdf:ode-S} with respect to $\alpha$ at $(\alpha,\varepsilon)=(0,0)$ (where $\theta_0(s)=s$, $k_0\equiv 1$, $v_0\equiv 0$)
gives the linear variational system
\begin{equation}\label{eq:cdf:firstVar}
\left\{
\begin{aligned}
x_1'&=-(\cos\theta_0)\,\theta_1=-\cos s\,\theta_1,\\
y_1'&=-(\sin\theta_0)\,\theta_1=-\sin s\,\theta_1,\\
\theta_1'&=k_1,\\
k_1'&=v_1,\\
v_1'&=\cos\theta_0=\cos s,
\end{aligned}
\right.
\qquad
x_1(0)=y_1(0)=\theta_1(0)=k_1(0)=v_1(0)=0,
\end{equation}
where in the last equation we used that at $(\alpha,\varepsilon)=(0,0)$ the term
$\alpha\varepsilon(x\cos\theta+y\sin\theta)$ contributes no first variation in $\alpha$.

Integrating \eqref{eq:cdf:firstVar} from the bottom yields
\[
v_1(s)=\int_0^s \cos t\,dt=\sin s,
\qquad
k_1(s)=\int_0^s v_1(t)\,dt=\int_0^s \sin t\,dt=1-\cos s,
\]
\begin{equation}
\label{eq:cdf:theta1-explicit}
\theta_1(s)=\int_0^s k_1(t)\,dt=\int_0^s (1-\cos t)\,dt=s-\sin s.
\end{equation}
Substituting $\theta_1$ into the first two equations of \eqref{eq:cdf:firstVar} and integrating gives explicit expressions
\begin{equation}\label{eq:cdf:x1y1}
x_1(s)=\int_0^s\bigl(-\cos t\,(t-\sin t)\bigr)\,dt
=-s\sin s+\frac12\sin^2 s-\cos s+1,
\end{equation}
\begin{equation}\label{eq:cdf:y1y1}
y_1(s)=\int_0^s\bigl(-\sin t\,(t-\sin t)\bigr)\,dt
=s\cos s+\frac{s}{2}-\sin s-\frac14\sin(2s).
\end{equation}
In particular, at $s=\pi$,
\begin{equation}\label{eq:cdf:endpointFirstVar}
\theta_1(\pi)=\pi,\qquad k_1(\pi)=2,\qquad x_1(\pi)=2,\qquad y_1(\pi)=-\frac{\pi}{2}.
\end{equation}
Differentiate \eqref{eq:cdf:PhiDef} in $\alpha$ and evaluate at $\varepsilon=0$:
\[
\Phi(\alpha,0,L)=-\int_0^L \cos\theta\,ds,
\qquad
\Phi_\alpha(\alpha,0,L)
=-\int_0^L(-\sin\theta\,\theta_\alpha)\,ds
=\int_0^L \sin\theta\,\theta_\alpha\,ds.
\]
At $\alpha=0$ we have $\theta=\theta_0=s$ and $\theta_\alpha=\theta_1=s-\sin s$, hence
\begin{align*}
\Phi_\alpha(0,0,\pi)
&=\int_0^\pi \sin s\,(s-\sin s)\,ds
=\int_0^\pi s\sin s\,ds-\int_0^\pi \sin^2 s\,ds.
\end{align*}
Integration by parts gives $\int_0^\pi s\sin s\,ds=\pi$, and $\int_0^\pi\sin^2 s\,ds=\pi/2$, hence
\[
\Phi_\alpha(0,0,\pi)=\pi-\frac{\pi}{2}=\frac{\pi}{2}.
\]
At $\varepsilon=0$ one has $B(\alpha,0,L)=-\sin\theta(L)$, so
\[
B_\alpha(\alpha,0,L)= -\cos\theta(L)\,\theta_\alpha(L).
\]
Evaluating at $(\alpha,L)=(0,\pi)$ gives
\[
B_\alpha(0,0,\pi)= -\cos\theta_0(\pi)\,\theta_1(\pi)= -\cos\pi\cdot \pi=\pi.
\]
Collecting the derivatives obtained finishes the proof.
\end{proof}

This immediately implies the following existence and uniqueness result.

\begin{proposition}\label{prop:cdf:fundamental-arc}
There exist $\varepsilon_0>0$ and unique $C^\infty$ functions
\[
\varepsilon\mapsto \alpha(\varepsilon),\qquad \varepsilon\mapsto L(\varepsilon),
\qquad |\varepsilon|<\varepsilon_0,
\]
with $\alpha(0)=0$, $L(0)=\pi$, such that
\begin{equation}\label{eq:cdf:Fzero}
F\bigl(\alpha(\varepsilon),\varepsilon,L(\varepsilon)\bigr)=(0,0).
\end{equation}
Moreover,
\begin{equation}\label{eq:cdf:first-order}
\alpha'(0)=-2,\qquad L'(0)=2\pi,
\end{equation}
so in particular $\alpha(\varepsilon)\neq 0$ for all sufficiently small $\varepsilon\neq 0$.
\end{proposition}

\begin{proof}
By Lemma~\ref{lem:cdf:jacobian}, $D_{(\alpha,L)}F(0,0,\pi)$ is invertible, hence the Implicit Function Theorem yields unique
$C^\infty$ functions $(\alpha(\varepsilon),L(\varepsilon))$ solving \eqref{eq:cdf:Fzero}.

Differentiating \eqref{eq:cdf:Fzero} at $\varepsilon=0$ gives
\[
F_\varepsilon+F_\alpha\,\alpha'(0)+F_L\,L'(0)=0.
\]
Using Lemma~\ref{lem:cdf:jacobian} this becomes
\[
\begin{bmatrix}
\frac{\pi}{2} & 1\\[1mm]
\pi & 1
\end{bmatrix}
\begin{bmatrix}\alpha'(0)\\ L'(0)\end{bmatrix}
=
\begin{bmatrix}\pi\\ 0\end{bmatrix},
\]
which yields \eqref{eq:cdf:first-order}.
\end{proof}

The family constructed in Proposition \ref{prop:cdf:fundamental-arc} consists of fundamental arcs.

\subsection{An abundance of angles}
\label{subsec:cdf:turning-angle}

As in the jellyfish construction, the dihedral closure requires control of the angle between the two
{radial seam lines} through the homothety centre.  In the present CDF setting the base arc is a semicircle,
so this seam angle is close to $\pi$ (rather than close to $0$ as in the jellyfish case).  It is therefore
convenient to work with the {angle deficit from $\pi$}.

Along the IFT branch from Proposition~\ref{prop:cdf:fundamental-arc} define
\[
\Theta(\varepsilon):=\theta\bigl(L(\varepsilon);\alpha(\varepsilon),\varepsilon\bigr),
\qquad
\overline{\theta}(\varepsilon):=\pi-\Theta(\varepsilon).
\]
Thus $\overline{\theta}(\varepsilon)$ is the acute angle between the radial line through $c_\varepsilon$ and $\gamma(0)$ and the
radial line through $c_\varepsilon$ and $\gamma(L(\varepsilon))$.

In order to create a wealth of examples, we require the endpoint tangent angle $\Theta$  to not be equal to $\pi$ along the family of fundamental arcs generated by Proposition \ref{prop:cdf:fundamental-arc}.
It turns out that $\Theta$ vanishes at first-order, which necessitates a lengthier calculation when compared to the previous case of jellyfish.
Nevertheless, it does move at second-order.

\begin{lemma}\label{lem:cdf:Theta-expansion}
Along the branch from Proposition~\ref{prop:cdf:fundamental-arc}, the endpoint tangent angle
admits the expansion
\begin{equation}\label{eq:cdf:Theta-expansion}
\Theta(\varepsilon)=\pi-\pi\,\varepsilon^2+O(\varepsilon^3)
\qquad(\varepsilon\to 0).
\end{equation}
In particular, $\Theta$ is continuous and non-constant in a neighbourhood of $\varepsilon=0$.
\end{lemma}

\begin{proof}
We work throughout at the base point
$(\alpha,\varepsilon,L)=(0,0,\pi)$. 

Differentiate ${\Theta}(\varepsilon)=\theta\bigl(L(\varepsilon);\alpha(\varepsilon),\varepsilon\bigr)$ once:
\begin{equation}\label{eq:cdf:ThetaPrime-chain}
{\Theta}'(0)
=
\theta_\varepsilon+\theta_\alpha\,\alpha'(0)+\theta_L\,L'(0)
\quad\text{evaluated at }(0,0,\pi).
\end{equation}
At $(0,0,\pi)$ we have $\theta_L=\theta_s(\pi)=k_0(\pi)=1$ and, from  Lemma~\ref{lem:cdf:jacobian},
\[
\theta_\alpha(0,0,\pi)=\theta_1(\pi)=\pi.
\]
Moreover, $\theta_\varepsilon(0,0,\pi)=0$ because $\varepsilon$ enters \eqref{eq:cdf:ode-S} only through the product
$\alpha\varepsilon$, and the base has $\alpha=0$.
Using \eqref{eq:cdf:first-order}, namely $\alpha'(0)=-2$ and $L'(0)=2\pi$, we obtain
\[
{\Theta}'(0)=0+\pi(-2)+1\cdot(2\pi)=0.
\]
To compute ${\Theta}''(0)$ we differentiate the identity $F(\alpha(\varepsilon),\varepsilon,L(\varepsilon))\equiv 0$ twice.
Write
\[
A:=D_{(\alpha,L)}F(0,0,\pi)
=
\begin{bmatrix}\Phi_\alpha & \Phi_L\\ B_\alpha & B_L\end{bmatrix}(0,0,\pi)
=
\begin{bmatrix}\frac{\pi}{2} & 1\\[0.7mm]\pi & 1\end{bmatrix}
\]
from Lemma~\ref{lem:cdf:jacobian}.
Differentiating once gives
\begin{equation}\label{eq:cdf:firstDiffF}
F_\varepsilon + F_\alpha\,\alpha'(0)+F_L\,L'(0)=0,
\end{equation}
which was already used to obtain \eqref{eq:cdf:first-order}.
Differentiating \eqref{eq:cdf:firstDiffF} once more yields, at $\varepsilon=0$,
\begin{equation}\label{eq:cdf:secondDiffF}
A
\begin{bmatrix}\alpha''(0)\\[0.3mm]L''(0)\end{bmatrix}
=
-\,S,
\end{equation}
where
\begin{equation}\label{eq:cdf:Sdef}
S:=
\Big(
F_{\varepsilon\varepsilon}
+2F_{\varepsilon\alpha}\alpha'(0)
+2F_{\varepsilon L}L'(0)
+F_{\alpha\alpha}\bigl(\alpha'(0)\bigr)^2
+2F_{\alpha L}\alpha'(0)L'(0)
+F_{LL}\bigl(L'(0)\bigr)^2
\Big)_{(0,0,\pi)} .
\end{equation}
We compute $\theta_{\varepsilon\alpha}$ by considering the $\alpha$-variation at $\alpha=0$ but with $\varepsilon$ free.
Differentiate the last equation in \eqref{eq:cdf:ode-S} with respect to $\alpha$ at $\alpha=0$:
\[
(\partial_\alpha v)' = \cos\theta_0 + \varepsilon\bigl(x_0\cos\theta_0+y_0\sin\theta_0\bigr)
= \cos s + \varepsilon,
\]
since $(x_0,y_0,\theta_0)=(\cos s,\sin s,s)$ and $x_0\cos\theta_0+y_0\sin\theta_0\equiv 1$.
With $v_1:=\partial_\alpha v|_{\alpha=0}$ and $v_1(0)=0$, this gives
\[
v_1(s)=\sin s+\varepsilon s.
\]
Integrating upward as in Lemma~\ref{lem:cdf:jacobian} gives
\[
k_1(s)=\int_0^s v_1(t)\,dt = 1-\cos s+\varepsilon\frac{s^2}{2},
\qquad
\theta_1(s)=\int_0^s k_1(t)\,dt = s-\sin s+\varepsilon\frac{s^3}{6}.
\]
Therefore
\begin{equation}\label{eq:cdf:theta_ea}
\theta_{\varepsilon\alpha}(0,0,\pi)
=
\partial_\varepsilon\theta_1(\pi)\Big|_{\varepsilon=0}
=
\frac{\pi^3}{6}.
\end{equation}

Introduce the auxiliary quantities
\[
P:=x\cos\theta+y\sin\theta,
\qquad
M:=y\cos\theta-x\sin\theta.
\]
Along the base arc, $P\equiv 1$ and $M\equiv 0$.

\paragraph{Mixed derivative of $\Phi$.}
From \eqref{eq:cdf:PhiDef},
\[
\Phi_\varepsilon
=
-\int_0^L\bigl(-\sin\theta\,\theta_\varepsilon + P+\varepsilon P_\varepsilon\bigr)\,ds.
\]
At the base point, $\theta_\varepsilon=0$ and $P\equiv 1$, so $\Phi_\varepsilon(0,0,\pi)=-\pi$.
Differentiate $\Phi_\varepsilon$ with respect to $\alpha$ and then evaluate at $(0,0,\pi)$:
\begin{align*}
\Phi_{\varepsilon\alpha}(0,0,\pi)
&=
-\int_0^\pi\Big(-\sin s\,\theta_{\varepsilon\alpha}(s) + P_\alpha(s)\Big)\,ds\\
&=
\int_0^\pi \sin s\,\theta_{\varepsilon\alpha}(s)\,ds
-\int_0^\pi P_\alpha(s)\,ds.
\end{align*}
At $\varepsilon=0$, the $\alpha$-variation computed in Lemma~\ref{lem:cdf:jacobian} gives
$x_1,y_1,\theta_1$ and hence
\[
P_\alpha = x_1\cos s+y_1\sin s.
\]
Using the explicit formulas \eqref{eq:cdf:x1y1}, \eqref{eq:cdf:y1y1}, one simplifies to
\begin{equation}\label{eq:cdf:Palpha}
P_\alpha(s)=\frac{s\sin s}{2}+\cos s-1.
\end{equation}
Consequently,
\[
\int_0^\pi P_\alpha(s)\,ds
=
\frac12\int_0^\pi s\sin s\,ds +\int_0^\pi \cos s\,ds-\int_0^\pi 1\,ds
=\frac12\pi+0-\pi=-\frac{\pi}{2}.
\]
Next, $\theta_{\varepsilon\alpha}(s)=s^3/6$ along the base, so
\[
\int_0^\pi \sin s\,\theta_{\varepsilon\alpha}(s)\,ds
=
\frac16\int_0^\pi s^3\sin s\,ds.
\]
Combining these computations gives
\begin{equation}\label{eq:cdf:Phi_ea}
\Phi_{\varepsilon\alpha}(0,0,\pi)
=
\frac{1}{6}\pi(\pi^2-6)+\frac{\pi}{2}
=
\frac{\pi(\pi^2-3)}{6}.
\end{equation}

\paragraph{Mixed derivative of $B$.}
From \eqref{eq:cdf:Bdef},
\[
B_\varepsilon
=
-\cos\theta(L)\,\theta_\varepsilon(L)+M(L)+\varepsilon M_\varepsilon(L),
\]
and at the base point $M(\pi)=0$ and $\theta_\varepsilon(\pi)=0$, hence $B_\varepsilon(0,0,\pi)=0$.
Differentiating once more in $\alpha$ gives
\begin{equation}\label{eq:cdf:Bea_split}
B_{\varepsilon\alpha}(0,0,\pi)
=
-\cos\pi\cdot \theta_{\varepsilon\alpha}(0,0,\pi)
+M_\alpha(0,0,\pi).
\end{equation}
We already have $\theta_{\varepsilon\alpha}(0,0,\pi)=\pi^3/6$ from \eqref{eq:cdf:theta_ea}.
For $M_\alpha$, note that at $\varepsilon=0$,
\[
M=y\cos\theta-x\sin\theta,
\]
so differentiating in $\alpha$ and evaluating at the base point gives
\[
M_\alpha(\pi)
=
y_1(\pi)\cos\pi - y_0(\pi)\sin\pi\,\theta_1(\pi)
-x_1(\pi)\sin\pi - x_0(\pi)\cos\pi\,\theta_1(\pi)= -y_1(\pi)-\theta_1(\pi).
\]
Using \eqref{eq:cdf:endpointFirstVar} gives $y_1(\pi)=-\pi/2$ and $\theta_1(\pi)=\pi$, hence
\[
M_\alpha(\pi)= -\Bigl(-\frac{\pi}{2}\Bigr)-\pi=-\frac{\pi}{2}.
\]
Substituting into \eqref{eq:cdf:Bea_split} yields
\begin{equation}\label{eq:cdf:B_ea}
B_{\varepsilon\alpha}(0,0,\pi)
=
\frac{\pi^3}{6}-\frac{\pi}{2}
=
\frac{\pi(\pi^2-3)}{6}.
\end{equation}
In particular, $\Phi_{\varepsilon\alpha}(0,0,\pi)=B_{\varepsilon\alpha}(0,0,\pi)$.

We next record the other second derivatives appearing in \eqref{eq:cdf:Sdef}.

\paragraph{The derivative $F_{\varepsilon L}$.}
From \eqref{eq:cdf:Phi-alpha0}, $\Phi(0,\varepsilon,L)=-\sin L-\varepsilon L$, hence
\[
\Phi_{\varepsilon L}(0,0,\pi)=\partial_L(-L)\big|_{L=\pi}=-1.
\]
Also, along the base circle $B(0,\varepsilon,L)\equiv -\sin L$ (since $M\equiv 0$ and $\theta=s$), so $B_{\varepsilon L}(0,0,\pi)=0$.
Therefore
\begin{equation}\label{eq:cdf:FepsL}
F_{\varepsilon L}(0,0,\pi)=\begin{bmatrix}-1\\[0.3mm]0\end{bmatrix}.
\end{equation}

\paragraph{The derivative $F_{\alpha L}$.}
At $\varepsilon=0$, $\Phi_\alpha=\int_0^L\sin\theta\,\theta_\alpha\,ds$, so
\[
\Phi_{\alpha L}(0,0,\pi)=\sin\theta_0(\pi)\,\theta_1(\pi)=0.
\]
Moreover, at $\varepsilon=0$, $B(\alpha,0,L)=-\sin\theta(L)$, hence
\[
B_{\alpha L}(0,0,\pi)
=
-\cos\theta_0(\pi)\,\theta_{\alpha L}(0,0,\pi)
=
(+1)\,\theta_{\alpha L}(0,0,\pi).
\]
Since $\theta_\alpha' = k_\alpha$, we have $\theta_{\alpha L}=\partial_L\theta_\alpha(L)=\theta_{\alpha s}(L)=k_\alpha(L)=k_1(L)$,
and \eqref{eq:cdf:endpointFirstVar} gives $k_1(\pi)=2$. Therefore
\begin{equation}\label{eq:cdf:FalphaL}
F_{\alpha L}(0,0,\pi)=\begin{bmatrix}0\\[0.3mm]2\end{bmatrix}.
\end{equation}

\paragraph{The derivative $F_{\varepsilon\varepsilon}$ and $F_{LL}$.}
At $\alpha=0$ both $\Phi(0,\varepsilon,L)$ and $B(0,\varepsilon,L)$ are affine in $\varepsilon$, hence
$F_{\varepsilon\varepsilon}(0,0,\pi)=0$.
Also, $\partial_{LL}(-\sin L)=\sin L$ vanishes at $L=\pi$, and the $L$-dependence through the solution does not contribute at $\alpha=0$,
so $F_{LL}(0,0,\pi)=0$.

\paragraph{The derivative $F_{\alpha\alpha}$.}
These are obtained by differentiating the variational system \eqref{eq:cdf:firstVar} once more in $\alpha$ and evaluating at $(0,0)$.
One finds the endpoint values
\begin{equation}\label{eq:cdf:Faa_values}
F_{\alpha\alpha}(0,0,\pi)
=
\begin{bmatrix}\Phi_{\alpha\alpha}\\[0.3mm]B_{\alpha\alpha}\end{bmatrix}(0,0,\pi)
=
\begin{bmatrix}
\frac{\pi(2\pi^2-39)}{12}\\[1mm]
\frac{\pi(2\pi^2-27)}{12}
\end{bmatrix}.
\end{equation}
The computation is standard but lengthy; we include the full derivation in Lemma \ref{lem:cdf:Faa}.

Using \eqref{eq:cdf:Faa_values} together with the identities already derived, we may now evaluate $S$.
Recall that $\alpha'(0)=-2$ and $L'(0)=2\pi$ from \eqref{eq:cdf:first-order}.
Let
\[
C:=\frac{\pi(\pi^2-3)}{6}.
\]
Then \eqref{eq:cdf:Phi_ea}-\eqref{eq:cdf:B_ea} give
\begin{equation}\label{eq:cdf:Fea}
F_{\varepsilon\alpha}(0,0,\pi)=\begin{bmatrix}C\\[0.3mm]C\end{bmatrix}.
\end{equation}
Substituting \eqref{eq:cdf:Fea}, \eqref{eq:cdf:FepsL}, \eqref{eq:cdf:FalphaL}, and \eqref{eq:cdf:Faa_values} into \eqref{eq:cdf:Sdef},
using $F_{\varepsilon\varepsilon}=F_{LL}=0$, yields
\[
S_1=-4C-4\pi+4\Phi_{\alpha\alpha},
\qquad
S_2=-4C-16\pi+4B_{\alpha\alpha}.
\]
Since
\[
4C=\frac{2}{3}\pi(\pi^2-3)=\frac{2}{3}\pi^3-2\pi,
\quad
4\Phi_{\alpha\alpha}=\frac{\pi(2\pi^2-39)}{3}=\frac{2}{3}\pi^3-13\pi,
\quad
4B_{\alpha\alpha}=\frac{\pi(2\pi^2-27)}{3}=\frac{2}{3}\pi^3-9\pi,
\]
we obtain
\[
S_1=\Bigl(-\frac{2}{3}\pi^3+2\pi\Bigr)-4\pi+\Bigl(\frac{2}{3}\pi^3-13\pi\Bigr)=-15\pi,
\qquad
S_2=\Bigl(-\frac{2}{3}\pi^3+2\pi\Bigr)-16\pi+\Bigl(\frac{2}{3}\pi^3-9\pi\Bigr)=-23\pi.
\]
Hence $S=(-15\pi,-23\pi)^{\mathsf T}$, and \eqref{eq:cdf:secondDiffF} becomes
\begin{equation}\label{eq:cdf:linear_system_a2l2}
\begin{bmatrix}\frac{\pi}{2}&1\\[0.5mm]\pi&1\end{bmatrix}
\begin{bmatrix}\alpha''(0)\\[0.3mm]L''(0)\end{bmatrix}
=
\begin{bmatrix}15\pi\\[0.3mm]23\pi\end{bmatrix}.
\end{equation}
Subtracting the second equation from twice the first gives $L''(0)=7\pi$, and substituting back yields $\alpha''(0)=16$.

Now we may calculate the second derivative ${\Theta}''(0)$.
Differentiate ${\Theta}(\varepsilon)=\theta(\alpha(\varepsilon),\varepsilon,L(\varepsilon))$ twice at $\varepsilon=0$:
\begin{align}\label{eq:cdf:ThetaSecond-chain}
{\Theta}''(0)
&=
\theta_{\alpha\alpha}\bigl(\alpha'(0)\bigr)^2
+2\theta_{\alpha L}\alpha'(0)L'(0)
+\theta_\alpha\,\alpha''(0)
+\theta_L\,L''(0)
+2\theta_{\varepsilon\alpha}\alpha'(0),
\end{align}
where we used that $\theta_\varepsilon(0,0,\pi)=0$ and $\theta_{\varepsilon L}(0,0,\pi)=\theta_{\varepsilon\varepsilon}(0,0,\pi)=0$
at the base point (again because $\varepsilon$ enters only through $\alpha\varepsilon$).
We now list the required values:
\[
\theta_\alpha(0,0,\pi)=\theta_1(\pi)=\pi,\qquad
\theta_L(0,0,\pi)=k_0(\pi)=1,\qquad
\theta_{\alpha L}(0,0,\pi)=k_1(\pi)=2,
\]
\[
\theta_{\varepsilon\alpha}(0,0,\pi)=\frac{\pi^3}{6},
\qquad
\alpha'(0)=-2,\quad L'(0)=2\pi,\quad \alpha''(0)=16,\quad L''(0)=7\pi.
\]
Finally, $\theta_{\alpha\alpha}(0,0,\pi)$ is obtained from the second $\alpha$-variation above:
since $\theta_{\alpha\alpha}(0,0,\pi)=\theta_2(\pi)$ and one computes $\theta_2(\pi)=\frac{\pi(2\pi^2-27)}{12}$,
we have
\begin{equation}\label{eq:cdf:theta_aa_value}
\theta_{\alpha\alpha}(0,0,\pi)=\frac{\pi(2\pi^2-27)}{12}.
\end{equation}
Substituting these values into \eqref{eq:cdf:ThetaSecond-chain} gives
\begin{align*}
{\Theta}''(0)
&=
\frac{\pi(2\pi^2-27)}{12}\cdot 4
+2\cdot 2\cdot(-2)\cdot(2\pi)
+\pi\cdot 16
+1\cdot(7\pi)
+2\cdot\frac{\pi^3}{6}\cdot(-2)\\
&=
\frac{\pi(2\pi^2-27)}{3}-16\pi+16\pi+7\pi-\frac{2\pi^3}{3}\\
&=
\Bigl(\frac{2\pi^3}{3}-9\pi\Bigr)+7\pi-\frac{2\pi^3}{3}
=-2\pi.
\end{align*}
Therefore ${\Theta}''(0)=-2\pi$, and the Taylor expansion gives
\[
{\Theta}(\varepsilon)={\Theta}(0)+\frac12{\Theta}''(0)\,\varepsilon^2+O(\varepsilon^3)
=\pi-\pi\varepsilon^2+O(\varepsilon^3),
\]
which is \eqref{eq:cdf:Theta-expansion}. In particular, ${\Theta}$ is continuous and non-constant near $\varepsilon=0$.
\end{proof}

\begin{lemma}[Second $\alpha$-variation of the boundary map]\label{lem:cdf:Faa}
Let $\Phi$ and $B$ be defined by \eqref{eq:cdf:PhiDef} and \eqref{eq:cdf:Bdef}, and let
$F=(\Phi,B)^{\mathsf T}$ as in \eqref{eq:cdf:Fdef}.  Then at the base point $(\alpha,\varepsilon,L)=(0,0,\pi)$ one has
\begin{equation}\label{eq:cdf:Faa_values}
F_{\alpha\alpha}(0,0,\pi)
=
\begin{bmatrix}\Phi_{\alpha\alpha}\\[0.3mm]B_{\alpha\alpha}\end{bmatrix}(0,0,\pi)
=
\begin{bmatrix}
\frac{\pi(2\pi^2-39)}{12}\\[1mm]
\frac{\pi(2\pi^2-27)}{12}
\end{bmatrix}.
\end{equation}
\end{lemma}

\begin{proof}
Recall the first $\alpha$-variation \eqref{eq:cdf:firstVar}.
Write
\[
(x_2,y_2,\theta_2,k_2,v_2):=\partial_{\alpha\alpha}(x,y,\theta,k,v)\big|_{(\alpha,\varepsilon)=(0,0)}.
\]
and differentiate \eqref{eq:cdf:ode-S} twice in $\alpha$, then evaluate at $(0,0)$ to find
\begin{equation}\label{eq:cdf:secondVar}
\left\{
\begin{aligned}
x_2'&=-\cos s\,\theta_2+\sin s\,\theta_1^2,\\
y_2'&=-\sin s\,\theta_2-\cos s\,\theta_1^2,\\
\theta_2'&=k_2,\\
k_2'&=v_2,\\
v_2'&=-2\sin s\,\theta_1,
\end{aligned}
\right.
\qquad
x_2(0)=y_2(0)=\theta_2(0)=k_2(0)=v_2(0)=0,
\end{equation}
where $\theta_1$ is given by \eqref{eq:cdf:theta1-explicit}.  (The factor $2$ in the last line comes from
differentiating $v'=\alpha\cos\theta$ twice: $(\alpha\cos\theta)_{\alpha\alpha}\big|_{\alpha=0}=-2\sin\theta_0\,\theta_1$.)

Integrating the last three equations in \eqref{eq:cdf:secondVar} (again from the bottom) and substituting
$\theta_1(s)=s-\sin s$ yields the explicit expressions
\begin{align}
v_2(s)
&=\int_0^s\!\bigl(-2\sin t\,(t-\sin t)\bigr)\,dt
=2s\cos s+s-2\sin s-\frac12\sin(2s),
\label{eq:cdf:v2-explicit}\\
k_2(s)
&=\int_0^s v_2(t)\,dt
=\frac{s^2}{2}+2s\sin s-\frac12\sin^2 s+4\cos s-4,
\label{eq:cdf:k2-explicit}\\
\theta_2(s)
&=\int_0^s k_2(t)\,dt
=\frac{s^3}{6}-2s\cos s-\frac{17}{4}s+6\sin s+\frac18\sin(2s).
\label{eq:cdf:theta2-explicit}
\end{align}
In particular,
\begin{equation}\label{eq:cdf:theta2-pi}
\theta_2(\pi)=\frac{\pi(2\pi^2-27)}{12}.
\end{equation}

When $\varepsilon=0$ we have $B(\alpha,0,L)=-\sin\theta(L;\alpha,0)$, hence
\[
B_{\alpha\alpha}(0,0,\pi)
=\Bigl(\sin\theta_0(\pi)\,\theta_1(\pi)^2-\cos\theta_0(\pi)\,\theta_2(\pi)\Bigr)
= -\cos\pi\cdot \theta_2(\pi)=\theta_2(\pi),
\]
since $\sin\theta_0(\pi)=\sin\pi=0$.  Using \eqref{eq:cdf:theta2-pi} gives
\[
B_{\alpha\alpha}(0,0,\pi)=\frac{\pi(2\pi^2-27)}{12}.
\]

When $\varepsilon=0$ the definition \eqref{eq:cdf:PhiDef} reduces to
\[
\Phi(\alpha,0,\pi)=-\int_0^\pi \cos\theta(s;\alpha,0)\,ds.
\]
Differentiating twice in $\alpha$ and evaluating at $\alpha=0$ (so $\theta_0(s)=s$) yields
\begin{equation}\label{eq:cdf:Phi_aa_integral}
\Phi_{\alpha\alpha}(0,0,\pi)
=\int_0^\pi \Bigl(\cos s\,\theta_1(s)^2+\sin s\,\theta_2(s)\Bigr)\,ds.
\end{equation}
We evaluate the two terms separately.

First, using $\theta_1=s-\sin s$,
\begin{align*}
\int_0^\pi \cos s\,\theta_1^2\,ds
&=\int_0^\pi \cos s\,(s-\sin s)^2\,ds
\\&=\int_0^\pi \bigl(s^2\cos s-2s\sin s\cos s+\sin^2 s\cos s\bigr)\,ds
\\&=-\frac{3\pi}{2}.
\end{align*}

Second, using \eqref{eq:cdf:theta2-explicit},
\begin{align*}
\int_0^\pi \sin s\,\theta_2(s)\,ds
&=\int_0^\pi \sin s\Bigl(\frac{s^3}{6}-2s\cos s-\frac{17}{4}s+6\sin s+\frac18\sin(2s)\Bigr)\,ds
\\&=\frac{\pi(\pi^2-6)}{6}+\frac{\pi}{2}-\frac{17\pi}{4}+3\pi
=\frac{\pi(2\pi^2-21)}{12}.
\end{align*}
Substituting into \eqref{eq:cdf:Phi_aa_integral} yields
\[
\Phi_{\alpha\alpha}(0,0,\pi)
=-\frac{3\pi}{2}+\frac{\pi(2\pi^2-21)}{12}
=\frac{\pi(2\pi^2-39)}{12}.
\]
Together with the computation of $B_{\alpha\alpha}(0,0,\pi)$ above, this proves \eqref{eq:cdf:Faa_values}.
\end{proof}

Now we may conclude the required abundance of angles.

\begin{corollary}\label{cor:cdf:eps-as-fn-of-theta}
There exist $\theta_\ast>0$ and $\varepsilon_\ast>0$ and unique $C^\infty$ functions
\[
\varepsilon=\varepsilon_\pm(\overline{\theta})
\qquad\text{for }\ \overline{\theta}\in(0,\theta_\ast),
\]
with $\varepsilon_+(\overline{\theta})\in(0,\varepsilon_\ast)$ and $\varepsilon_-(\overline{\theta})\in(-\varepsilon_\ast,0)$, such that
\[
\pi-\theta\bigl(L(\varepsilon_\pm(\overline{\theta}));\alpha(\varepsilon_\pm(\overline{\theta})),\varepsilon_\pm(\overline{\theta})\bigr)
=\overline{\theta}.
\]
Moreover,
\begin{equation}\label{eq:cdf:eps-asymp}
\varepsilon_\pm(\overline{\theta})
=
\pm\sqrt{\frac{\overline{\theta}}{\pi}}
+O(\overline{\theta})
\qquad(\overline{\theta}\to 0^+).
\end{equation}
In particular, for every integer $q$ sufficiently large there exists $\varepsilon_q\in(0,\varepsilon_\ast)$ such that
\begin{equation}\label{eq:cdf:Theta_qminus1_over_q}
\Theta(\varepsilon_q)=\pi-\frac{\pi}{q}=\frac{q-1}{q}\,\pi.
\end{equation}
\end{corollary}

\begin{proof}
By Lemma~\ref{lem:cdf:theta-bar-expansion}, the map $\varepsilon\mapsto\overline{\theta}(\varepsilon)$ is strictly monotone on each side of $0$.
Hence the restrictions $\overline{\theta}:(0,\varepsilon_\ast)\to(0,\theta_\ast)$ and $\overline{\theta}:(-\varepsilon_\ast,0)\to(0,\theta_\ast)$
are bijections for $\varepsilon_\ast$ small, and each has a $C^\infty$ inverse by the Inverse Function Theorem
(applied at any $\varepsilon\neq 0$, using \eqref{eq:cdf:theta-bar-derivative}).
This yields the inverses $\varepsilon_\pm(\overline{\theta})$.

For the asymptotic, rewrite \eqref{eq:cdf:theta-bar-expansion} as
$\overline{\theta}=\pi\varepsilon^2(1+O(\varepsilon))$ and solve for $\varepsilon$ on each side, giving \eqref{eq:cdf:eps-asymp}.
Finally, take $\overline{\theta}=\pi/q$ and set $\varepsilon_q:=\varepsilon_+(\pi/q)$ to obtain \eqref{eq:cdf:Theta_qminus1_over_q}.
\end{proof}

\subsection{Dihedral gluing and proof of Theorem~\ref{thm:intro:cdf-epi}}
\label{subsec:cdf:dihedral-gluing}

We now explain how fundamental arcs yield closed epicyclic shrinkers for curve diffusion flow.
The construction is formally identical to the jellyfish case (Subsection~\ref{subsec:dihedral-gluing}),
so we only record the genuinely CDF-specific input: the seam condition is expressed via the {radial seam functional} $B$
and the degenerate endpoint condition $v(L)=0$ has been replaced by $\Phi=0$ (Definition~\ref{def:cdf:fundamental-arc}).
We present Figure \ref{fig:cdf:4_5} for a visualisation of the procedure, and for a different example, the middle plot in Figure \ref{fig:intro:representatives}

\begin{figure}[t]
\centering

\begin{minipage}[t]{0.39\textwidth}
  \vspace{0pt}\centering 

  \begin{subfigure}[t]{\linewidth}
    \centering
    \incfig{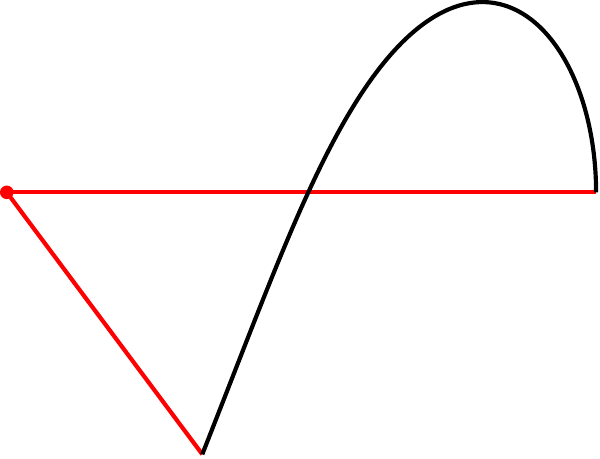}{0.61\linewidth}{0}
    \caption{Fundamental arc.}
  \end{subfigure}

  \vspace{2mm}

  \begin{subfigure}[t]{\linewidth}
    \centering
    \incfig{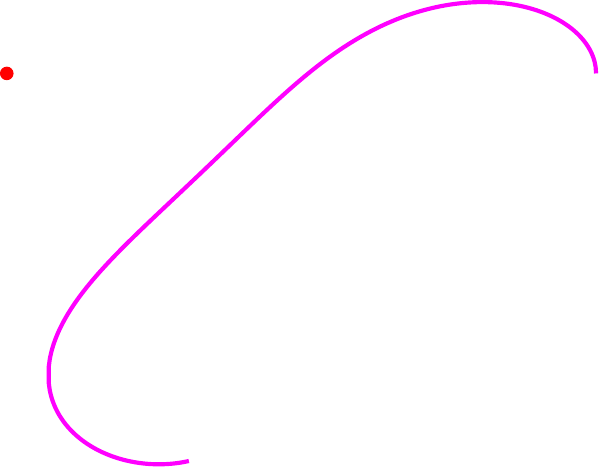}{0.61\linewidth}{0}
    \caption{Doubled arc.}
  \end{subfigure}
\end{minipage}
\begin{minipage}[t]{0.40\textwidth}
  \vspace{0pt}\centering 

  \begin{subfigure}[t]{\linewidth}
    \centering
    \incfig{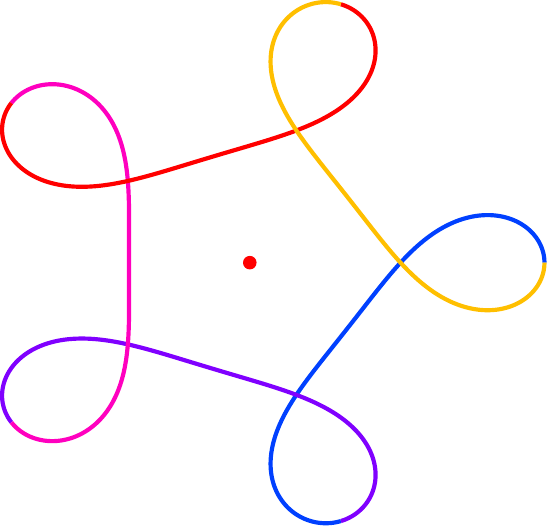}{\linewidth}{0}
    \caption{Closed curve (5 doubled arcs).}
  \end{subfigure}
\end{minipage}

\caption{Curve diffusion flow (CDF) epicyclic shrinker for the dihedral gluing data $p/q=4/5$:
$\omega=4$ and $\varepsilon=0.19672$ (to five significant figures). The closed curve is assembled from $5$ doubled arcs.}
\label{fig:cdf:4_5}
\end{figure}

\subsubsection{Reflection principle and dihedral concatenation}

\begin{lemma}\label{lem:reflection-principle-cdf}
Let $\gamma:[0,L]\to\R^2$ be an arc-length parametrised solution of the homothetic shrinker equation
\eqref{eq:intro:homot} on $(0,L)$. Assume that $\gamma(L)\in\ell$, that $\gamma$ meets $\ell$ orthogonally at $s=L$,
and that $k_s(L)=0$. Then the reflected arc
\[
\widetilde\gamma(s):=\sigma_\ell\bigl(\gamma(2L-s)\bigr),
\qquad s\in[L,2L],
\]
glues to $\gamma$ to give a $C^\infty$ solution on $[0,2L]$. The same holds at $s=0$.
\end{lemma}

\begin{proof}
The proof is similar to  Lemma~\ref{lem:reflection-principle-jellyfish}.
\end{proof}

\begin{remark}
\label{eq:cdf:B-is-RT}
Fix $\overline{\theta}\in(0,\theta_\ast)$ and choose $\varepsilon=\varepsilon_+(\overline{\theta})$ from
Corollary~\ref{cor:cdf:eps-as-fn-of-theta}, together with the corresponding fundamental arc
$\gamma:[0,L]\to\R^2$ given by Proposition~\ref{prop:cdf:fundamental-arc}.
Set $c_\varepsilon:=(-1/\varepsilon,0)$ and let $\ell_0$ (resp.\ $\ell_1$) be the line through $c_\varepsilon$ and $\gamma(0)$
(resp.\ through $c_\varepsilon$ and $\gamma(L)$).
The endpoint condition $B=0$
is precisely the radial orthogonality condition $R(L)\cdot T(L)=0$, i.e.\ $\gamma$ meets $\ell_1$ orthogonally at $s=L$.
Moreover, since $\Phi=0$ and $\alpha\neq 0$, we have $v(L)=k_s(L)=0$ by \eqref{eq:cdf:vPhi}. The same conditions hold at $s=0$
by the normalisation $v(0)=0$ and the fact that $R(0)$ is radial.

Define the doubled arc $\Gamma:[0,2L]\to\R^2$ by reflecting $\gamma$ across $\ell_1$ at $s=L$:
\[
\Gamma(s):=
\begin{cases}
\gamma(s), & s\in[0,L],\\
\sigma_{\ell_1}\!\big(\gamma(2L-s)\big), & s\in[L,2L].
\end{cases}
\]
By Lemma~\ref{lem:reflection-principle-cdf}, $\Gamma$ is $C^\infty$ and satisfies \eqref{eq:intro:homot} on $[0,2L]$.
\end{remark}

Let $\rho:=\sigma_{\ell_1}\circ\sigma_{\ell_0}$. Since $\ell_0$ and $\ell_1$ are distinct lines meeting at $c_\varepsilon$,
$\rho$ is a rotation about $c_\varepsilon$ through angle $2\Theta$, where $\Theta=\theta(L)$ is the seam angle between $\ell_0$ and $\ell_1$.
In particular, if $\Theta=\frac{p}{q}\pi$ with $\gcd(p,q)=1$, then $\rho^q=\mathrm{Id}$.

\begin{proposition}\label{prop:cdf:dihedral-closure}
Assume the seam angle satisfies
\[
\frac{\Theta}{\pi}\in\Q,
\qquad\text{equivalently}\qquad
\Theta=\frac{p}{q}\,\pi \ \text{ for coprime integers }p,q\ge 1.
\]
Then the concatenation
\[
\gamma_{p/q}:=\Gamma_0\#\Gamma_1\#\cdots\#\Gamma_{q-1},
\qquad
\Gamma_j:=\rho^j\circ\Gamma,
\]
defines a $C^\infty$ closed immersed curve $\gamma_{p/q}:\S\to\R^2$ satisfying \eqref{eq:intro:homot}.
Moreover, $\gamma_{p/q}$ is invariant under the dihedral group generated by $\rho$ and $\sigma_{\ell_0}$, hence has dihedral symmetry
of order $q$.
\end{proposition}

\begin{proof}
This is similar to the proof of Proposition~\ref{prop:dihedral-closure} in the jellyfish section:
each $\Gamma_j$ is a rigid motion image of $\Gamma$ and hence satisfies \eqref{eq:intro:homot};
seam matching follows from the reflection principle and the fact that $\rho$ maps $\ell_0$ to $\ell_1$;
and $\rho^q=\mathrm{Id}$ yields closure.
\end{proof}

\subsection{Maximality of the symmetry order and geometric distinctness}
\label{subsec:cdf:distinctness}

As in the jellyfish case, geometric distinctness follows by showing that there are no additional ``hidden'' reflection axes.
For curve diffusion flow, the decisive obstruction is the absence of {interior radial seams} on the fundamental arc.

\begin{lemma}\label{lem:cdf:no-interior-seams}
Let $|\varepsilon|$ be sufficiently small and let $\gamma:[0,L(\varepsilon)]\to\R^2$ be the fundamental arc from
Proposition~\ref{prop:cdf:fundamental-arc}. Then the radial seam function
\[
s\longmapsto
-\sin\theta(s)+\varepsilon\bigl(y(s)\cos\theta(s)-x(s)\sin\theta(s)\bigr)
\]
vanishes only at $s=0$ and $s=L(\varepsilon)$. Equivalently, $R(s)\cdot T(s)\neq 0$ for all $s\in(0,L(\varepsilon))$.
\end{lemma}

\begin{proof}
For the base semicircle \eqref{eq:cdf:circle} at $(\alpha,\varepsilon)=(0,0)$, the seam function equals $-\sin s$,
which has exactly two simple zeros on $[0,\pi]$ (at $s=0$ and $s=\pi$) and is strictly negative on $(0,\pi)$.

By smooth dependence of solutions on parameters (as in Lemma~\ref{lem:smooth-dependence} of the jellyfish section),
the seam function depends continuously on $(s,\alpha,\varepsilon)$ on compact subsets.
Hence, for $|\varepsilon|$ sufficiently small, the two simple zeros persist near $0$ and near $L(\varepsilon)$,
and no new zero can appear unless a multiple root forms, i.e.\ the seam function and its $s$-derivative vanish simultaneously.
Since $-\sin s$ has no multiple root on $[0,\pi]$, such a bifurcation is excluded for small perturbations.
Therefore there are no interior zeros.
\end{proof}

\begin{lemma}\label{lem:cdf:no-extra-axes}
Let $\gamma_{p/q}$ be a closed curve obtained by the dihedral gluing in Proposition~\ref{prop:cdf:dihedral-closure}.
Then the full symmetry group of $\gamma_{p/q}$ is exactly the dihedral group $D_q$; in particular,
$\gamma_{p/q}$ does {not} admit dihedral symmetry of order $\hat q$ for any $\hat q>q$.
\end{lemma}

\begin{proof}
The curve is invariant under $D_q$ by construction.
If it had dihedral symmetry of order $\hat q>q$ about the same centre $c_\varepsilon$, then there would exist a reflection axis
$\ell$ strictly between two adjacent seam rays used in the construction.
At any fixed point of a reflection, the curve meets the axis orthogonally; in centre coordinates this is exactly the radial seam
condition $R\cdot T=0$.
Such an interior seam would occur on the fundamental arc portion between adjacent seam rays, contradicting
Lemma~\ref{lem:cdf:no-interior-seams}.
\end{proof}

For epicyclic shrinkers it is convenient to use two similarity invariants: the dihedral symmetry order $q$ and the
rotation index (turning number), determined by the rational parameter $\Theta=\frac{p}{q}\pi$.

\begin{lemma}[Rotation index]\label{lem:cdf:turning-number}
Let $\gamma_{p/q}$ be the closed curve from Proposition~\ref{prop:cdf:dihedral-closure} with $\Theta=\frac{p}{q}\pi$.
Then the total turning of the tangent is $2q\Theta=2p\pi$. In particular, the rotation index of $\gamma_{p/q}$ equals $p$.
\end{lemma}

\begin{proof}
Along the fundamental arc the tangent angle increases by $\Theta$, doubling increases it by $2\Theta$, and concatenating $q$ rotated copies
increases it by $2q\Theta=2p\pi$. The rotation index is total turning divided by $2\pi$.
\end{proof}

\begin{corollary}[Non-equivalence of epicyclic shrinkers]\label{cor:cdf:non-equivalence}
Let $\gamma_{p/q}$ and $\gamma_{\hat p/\hat q}$ be closed epicyclic shrinkers produced by Proposition~\ref{prop:cdf:dihedral-closure}
with coprime pairs $(p,q)$ and $(\hat p,\hat q)$. If $(p,q)\neq(\hat p,\hat q)$, then $\gamma_{p/q}$ is not equivalent to
$\gamma_{\hat p/\hat q}$ under any similarity transformation.
\end{corollary}

\begin{proof}
Similarities preserve the rotation index and conjugate the full symmetry group.
If $p\neq \hat p$, then Lemma~\ref{lem:cdf:turning-number} rules out similarity equivalence.
If $p=\hat p$ but $q\neq \hat q$, then Lemma~\ref{lem:cdf:no-extra-axes} implies the full symmetry groups are $D_q$ and $D_{\hat q}$,
which have different orders $2q\neq 2\hat q$ and hence cannot be conjugate. Therefore no similarity can map one curve to the other.
\end{proof}

\begin{proof}[Proof of Theorem~\ref{thm:intro:cdf-epi}]
Choose $q_0$ so large that $\pi/q_0<\theta_\ast$ (from Corollary~\ref{cor:cdf:eps-as-fn-of-theta}).
For each integer $q\ge q_0$, set $\overline{\theta}=\pi/q$ and obtain $\varepsilon_q$ with
$\Theta(\varepsilon_q)=\frac{q-1}{q}\pi$ by \eqref{eq:cdf:Theta_qminus1_over_q}.
Then Proposition~\ref{prop:cdf:dihedral-closure} produces a smooth closed curve $\gamma_{(q-1)/q}$ with dihedral symmetry of order $q$
satisfying the homothetic shrinker equation \eqref{eq:intro:homot}, hence a self-similarly shrinking solution of curve diffusion flow.
Geometric distinctness is proved in Corollary \ref{cor:cdf:non-equivalence}.
\end{proof}

\section{Epicyclic expanders for the ideal flow}
\label{sec:ideal-epicycles}

In this section we prove Theorem~\ref{thm:intro:if-epi}.  The strategy follows the curve-diffusion epicyclic
shrinkers construction (Section~\ref{sec:cdf-epicycles}): we (i) produce a one-parameter family of
{fundamental arcs} by an Implicit Function Theorem, (ii) show the terminal turning angle varies over an interval
and hence attains infinitely many rational multiples of $\pi$, and (iii) glue by the same reflection/dihedral
concatenation to obtain infinitely many closed homothetic solutions.  The  genuinely new feature is that the ideal
flow is two orders higher, so the seam conditions require matching higher curvature jets.

\subsection{A complex identity and a first integral}
\label{subsec:ideal:complex-identity}

We identify $\R^2\simeq\C$ and use the {normal-angle} convention from the previous sections:
\[
N=e^{i\theta}=(\cos\theta,\sin\theta),\qquad \theta_s=k,
\qquad
T=iN=(-\sin\theta,\cos\theta),
\qquad
\gamma_s=T,\ \ \gamma=x+iy.
\]
Define the curvature quantities
\[
M:=k_{sss},\qquad N:=kk_{ss}-\frac12 k_s^2,
\qquad
\SQ:=(M+iN)e^{-i\theta},
\]
and the operator
\[
\SF[k] = k_{s^4} + k_{ss}k^2 - \frac12 k_s^2 k\,.
\]

\begin{lemma}\label{lem:ideal:Qs}
One has
\begin{equation}\label{eq:ideal:Qs}
\partial_s\SQ=\SF[k]\,e^{-i\theta}.
\end{equation}
\end{lemma}

\begin{proof}
A direct computation gives $N_s=kM$ and $\SF[k]=M_s+kN$.
Differentiating $\SQ=(M+iN)e^{-i\theta}$ and using $\theta_s=k$ yields
\[
\SQ_s=(M_s+iN_s-ik(M+iN))e^{-i\theta}=(M_s+kN)e^{-i\theta}=\SF[k]e^{-i\theta}.
\]
\end{proof}

\subsection{The fundamental arc}
\label{subsec:ideal:fundamental-arc}

\subsubsection{The homothetic ODE system and parameters}

Fix parameters $(\alpha,\varepsilon)$ near $(0,0)$.
As in the previous epicyclic sections, the (distant) epicycle centre is
\[
c_\varepsilon:=\Bigl(-\frac{1}{\varepsilon},0\Bigr).
\]
We introduce the curvature jet variables
\[
p:=k_s,\qquad q:=k_{ss},\qquad M:=k_{sss},\qquad N:=kq-\frac12p^2,
\qquad \SQ:=(M+iN)e^{-i\theta}.
\]
For homothetic solutions we prescribe the first integral in the form
\begin{equation}\label{eq:ideal:Q_s_homothetic}
\SQ_s
=
-\alpha\Bigl(\cos\theta+\varepsilon(x\cos\theta+y\sin\theta)\Bigr)e^{-i\theta}
\qquad(\Longleftrightarrow\ \SF[k]=-\alpha(\cos\theta+\varepsilon\,\gamma\cdot N)).
\end{equation}

\subsubsection{A $7$-equation first-order system}

We take state variables
\[
S(s;\alpha,\varepsilon,b):=(x,y,\theta,k,p,\SQ)\in\R^5\times\C
\qquad (7\text{ scalar ODEs}),
\]
and define auxiliary real variables via
\begin{equation}\label{eq:ideal:MN_from_Q}
M:=\Re(\SQ e^{i\theta}),\qquad N:=\Im(\SQ e^{i\theta}),
\qquad
q:=\frac{N+\frac12p^2}{k}.
\end{equation}
(These are well-defined provided $k\neq 0$, which holds on a uniform interval for parameters near the base semicircle.)

The homothetic shooting system is
\begin{equation}\label{eq:ideal:homothetic_xyQ_ode}
\left\{
\begin{aligned}
x_s&=-\sin\theta,\\
y_s&=\cos\theta,\\
\theta_s&=k,\\
k_s&=p,\\
p_s&=q=\dfrac{\Im(\SQ e^{i\theta})+\frac12p^2}{k},\\[2mm]
\SQ_s&=
-\alpha\Bigl(\cos\theta+\varepsilon(x\cos\theta+y\sin\theta)\Bigr)e^{-i\theta}.
\end{aligned}
\right.
\end{equation}

\begin{lemma}\label{lem:ideal:xyQ_consistency}
Let $(x,y,\theta,k,p,\SQ)$ solve \eqref{eq:ideal:homothetic_xyQ_ode} on an interval where $k\neq 0$,
and define $M,N,q$ by \eqref{eq:ideal:MN_from_Q}. Then $p=k_s$, $q=k_{ss}$, $M=k_{sss}$, and the curvature satisfies
\[
\SF[k]=-\alpha\Bigl(\cos\theta+\varepsilon(x\cos\theta+y\sin\theta)\Bigr)
\qquad\text{on the interval.}
\]
\end{lemma}

\begin{proof}
From \eqref{eq:ideal:homothetic_xyQ_ode} we have $k_s=p$ and $p_s=q$, hence $p=k_s$ and $q=k_{ss}$.
Write $\SQ e^{i\theta}=M+iN$. Differentiating and using $\theta_s=k$ gives
\[
(M+iN)_s=\SQ_s e^{i\theta} + i\theta_s \SQ e^{i\theta}
= -\alpha\Bigl(\cos\theta+\varepsilon(x\cos\theta+y\sin\theta)\Bigr) + i k (M+iN),
\]
so $N_s=kM$. Differentiating $q=\frac{N+\frac12p^2}{k}$ and using $k_s=p$, $p_s=q$ yields $q_s=M$, hence $M=k_{sss}$.
Finally \eqref{eq:ideal:Q_s_homothetic} and Lemma~\ref{lem:ideal:Qs} imply the claimed identity for $\SF[k]$.
\end{proof}

\subsubsection{Initial conditions and endpoint seam data}

We normalise as for the base semicircle:
\begin{equation}\label{eq:ideal:IC}
x(0)=1,\qquad y(0)=0,\qquad \theta(0)=0,\qquad k(0)=1,\qquad p(0)=0,\qquad \SQ(0)= i b,
\end{equation}
where $b\in\R$ is a shooting parameter (so $M(0)=0$ and $N(0)=b$, hence $q(0)=b$).

Let $L>0$ be the (unknown) arc length. We impose the endpoint seam conditions
\begin{equation}\label{eq:ideal:endpoint_conditions}
V(\alpha,\varepsilon,b,L):=p(L)=0,
\qquad
U(\alpha,\varepsilon,b,L):=M(L)=\Re(\SQ e^{i\theta})(L)=0,
\qquad
B(\alpha,\varepsilon,b,L)=0,
\end{equation}
where $B$ is the same radial seam functional used previously (now based at $c_\varepsilon$):
\begin{equation}\label{eq:ideal:Bdef}
B(\alpha,\varepsilon,b,L)
:=
-\sin\theta(L)
+\varepsilon\bigl(y(L)\cos\theta(L)-x(L)\sin\theta(L)\bigr).
\end{equation}
Indeed, with $R:=\gamma-c_\varepsilon$ one computes as in Remark \ref{eq:cdf:B-is-RT} that
\[
R(L)\cdot T(L)=\frac{1}{\varepsilon}\,B(\alpha,\varepsilon,b,L),
\]
so $B=0$ is equivalent to orthogonality of the tangent to the radius from $c_\varepsilon$ at the endpoint.

\subsubsection{A reduced endpoint functional}

As in Section~\ref{sec:cdf-epicycles}, the base family at $\alpha=0$ forces a degeneracy.
We therefore replace the condition $V=0$ by a smooth reduced equation obtained by dividing out the trivial
$\alpha$-factor.  Set
\[
G(\alpha,\varepsilon,b,L):=V(\alpha,\varepsilon,b,L)+U(\alpha,\varepsilon,b,L),
\]
and define
\begin{equation}\label{eq:ideal:Phi_def}
\Phi(\alpha,\varepsilon,b,L)
:=
\begin{cases}
\dfrac{G(\alpha,\varepsilon,b,L)-G(0,\varepsilon,b,L)}{\alpha}, & \alpha\neq 0,\\[2mm]
\partial_\alpha G(0,\varepsilon,b,L), & \alpha=0.
\end{cases}
\end{equation}
Smooth dependence of ODE solutions on parameters implies that $\Phi$ is $C^\infty$.

We take as endpoint map
\begin{equation}\label{eq:ideal:F_def}
F(\alpha,\varepsilon,b,L):=
\begin{bmatrix}
\Phi(\alpha,\varepsilon,b,L)\\[0.5mm]
U(\alpha,\varepsilon,b,L)\\[0.5mm]
B(\alpha,\varepsilon,b,L)
\end{bmatrix}.
\end{equation}

\begin{definition}\label{def:ideal:fundamental-arc}
A \emph{fundamental arc for the ideal epicycle problem} is a solution of
\eqref{eq:ideal:homothetic_xyQ_ode}, \eqref{eq:ideal:IC} for which there exist $L>0$ and $\alpha\neq 0$ such that
\[
F(\alpha,\varepsilon,b,L)=0.
\]
Equivalently, \eqref{eq:ideal:endpoint_conditions} holds with $\alpha\neq 0$.
\end{definition}

\subsection{Base semicircle and the reduced equation at $\alpha=0$}
\label{subsec:ideal:base}

\begin{lemma}\label{lem:ideal:base-semicircle}
At $(\alpha,\varepsilon,b)=(0,0,0)$ the unique solution of
\eqref{eq:ideal:homothetic_xyQ_ode}, \eqref{eq:ideal:IC} 
\begin{equation}\label{eq:ideal:base_semicircle}
x(s)=\cos s,\qquad y(s)=\sin s,\qquad \theta(s)=s,\qquad k\equiv 1,\qquad p\equiv 0,\qquad \SQ\equiv 0.
\end{equation}
In particular, with $L_0:=\pi$ one has $F(0,0,0,L_0)=0$.
\end{lemma}

\begin{proof}
At $\alpha=0$ the equation for $\SQ$ in \eqref{eq:ideal:homothetic_xyQ_ode} gives $\SQ\equiv ib$,
so with $b=0$ we have $\SQ\equiv 0$, hence $M=N\equiv 0$ and $q\equiv 0$. Therefore $p\equiv 0$, $k\equiv 1$,
$\theta(s)=s$ and $(x,y)=(\cos s,\sin s)$ by integration of $x_s=-\sin s$, $y_s=\cos s$ with $(x(0),y(0))=(1,0)$.
The endpoint conditions at $L_0=\pi$ are immediate.
\end{proof}

We note the following explicit formula for $\Phi$ at $\alpha=0$.

\begin{lemma}\label{lem:ideal:Phi-alpha0}
For $(\varepsilon,L)$ near $(0,\pi)$ one has
\begin{equation}\label{eq:ideal:Phi_alpha0_formula}
\Phi(0,\varepsilon,0,L)=-(\sin L+\varepsilon L).
\end{equation}
\end{lemma}

\begin{proof}
Recall that $G:=V+U$ with $V(\alpha,\varepsilon,b,L)=p(L)=k_s(L)$ and
$U(\alpha,\varepsilon,b,L)=M(L)=k_{sss}(L)$ (Lemma~\ref{lem:ideal:xyQ_consistency}). 
For $b=0$ we have (using \eqref{eq:ideal:Phi_def})
\begin{equation}\label{eq:Phi_alpha0_is_Galpha}
\Phi(0,\varepsilon,0,L)=\partial_\alpha G(0,\varepsilon,0,L)
=\partial_\alpha\bigl(p(L)+M(L)\bigr)\Big|_{\alpha=0,b=0}.
\end{equation}

Fix $\varepsilon$ and consider the solution of the homothetic system at $\alpha=0$ and $b=0$.
Since the equation for $\SQ$ in \eqref{eq:ideal:homothetic_xyQ_ode} is multiplied by $\alpha$, at $\alpha=0$ it reduces to
$\SQ_s\equiv 0$; with $\SQ(0)=ib=0$ we obtain $\SQ\equiv 0$, hence $M\equiv N\equiv 0$ and therefore $q\equiv 0$.
Consequently $p\equiv 0$, $k\equiv 1$, $\theta(s)=s$, and $\gamma(s)=(\cos s,\sin s)$ on the common existence interval
(Lemma~\ref{lem:ideal:base-semicircle}). In particular, along this base solution the endpoint function $G(0,\varepsilon,0,L)$
vanishes identically.

To compute the $\alpha$-derivative in \eqref{eq:Phi_alpha0_is_Galpha}, we write the curvature jets
\[
v:=k_s,\qquad w:=k_{ss},\qquad u:=k_{sss}.
\]
From Lemma~\ref{lem:ideal:Qs} and the definition $\SQ=(u+i(kw-\tfrac12v^2))e^{-i\theta}$, the curvature equation for a homothetic
solution can be written purely in real form as
\begin{equation}\label{eq:ideal:real_jet_eq}
u_s+k\Bigl(kw-\tfrac12v^2\Bigr)=-\alpha\Bigl(\cos\theta+\varepsilon(x\cos\theta+y\sin\theta)\Bigr),
\end{equation}
together with the jet identities
\begin{equation}\label{eq:ideal:jet_identities}
v_s=w,\qquad w_s=u.
\end{equation}
Differentiate \eqref{eq:ideal:real_jet_eq}-\eqref{eq:ideal:jet_identities} with respect to $\alpha$ at $\alpha=0$ (with $\varepsilon$
fixed) along the base semicircle. Denoting $\partial_\alpha|_{\alpha=0}$ by a subscript $\alpha$, we obtain
\begin{equation}\label{eq:ideal:var_system}
v_\alpha'=w_\alpha,\qquad w_\alpha'=u_\alpha,\qquad
u_\alpha'+w_\alpha= -\Bigl(\cos\theta+\varepsilon(x\cos\theta+y\sin\theta)\Bigr)\Big|_{\text{base}},
\end{equation}
with initial conditions $v_\alpha(0)=w_\alpha(0)=u_\alpha(0)=0$ (since $v(0)=w(0)=u(0)=0$ for all $\alpha$ by the initial data).
On the base semicircle, $\theta(s)=s$ and $x\cos\theta+y\sin\theta\equiv 1$, hence the forcing in \eqref{eq:ideal:var_system} is
$-(\cos s+\varepsilon)$, and therefore
\begin{equation}\label{eq:ideal:forced_oscillator}
w_\alpha''+w_\alpha=-(\cos s+\varepsilon),\qquad w_\alpha(0)=0,\quad w_\alpha'(0)=u_\alpha(0)=0.
\end{equation}

We solve \eqref{eq:ideal:forced_oscillator} explicitly.
A resonant particular solution for $-\cos s$ is $-\tfrac12 s\sin s$, and a particular solution for the constant forcing
$-\varepsilon$ is $-\varepsilon(1-\cos s)$. Both satisfy the initial conditions, hence
\begin{equation}\label{eq:ideal:walpha_explicit}
w_\alpha(s)= -\frac12 s\sin s-\varepsilon(1-\cos s).
\end{equation}
Integrating $v_\alpha'=w_\alpha$ with $v_\alpha(0)=0$ and differentiating $u_\alpha=w_\alpha'$ yield
\begin{align}
v_\alpha(s)
&=\int_0^s w_\alpha(t)\,dt
= -\frac12(\sin s-s\cos s)-\varepsilon(s-\sin s), \label{eq:ideal:valpha_explicit}\\[1mm]
u_\alpha(s)
&=w_\alpha'(s)
= -\frac12(\sin s+s\cos s)-\varepsilon\sin s. \label{eq:ideal:ualpha_explicit}
\end{align}
Therefore, for any $L$ in the common existence interval,
\begin{equation}\label{eq:ideal:Galphacalc}
\partial_\alpha\bigl(p(L)+M(L)\bigr)\Big|_{\alpha=0,b=0}
=v_\alpha(L)+u_\alpha(L)
= -\sin L-\varepsilon L.
\end{equation}
Combining \eqref{eq:Phi_alpha0_is_Galpha} and \eqref{eq:ideal:Galphacalc} gives \eqref{eq:ideal:Phi_alpha0_formula}.
\end{proof}

\subsection{Construction of a fundamental arc}
\label{subsec:ideal:IFT}

\begin{lemma}\label{lem:ideal:Jacobian}
At $(\alpha,\varepsilon,b,L)=(0,0,0,\pi)$ the Jacobian of $F$ with respect to $(\alpha,b,L)$ is invertible.
More precisely,
\begin{equation}\label{eq:ideal:Jacobian}
D_{(\alpha,b,L)}F(0,0,0,\pi)
=
\begin{bmatrix}
\Phi_\alpha(0,0,0,\pi) & \frac{5\pi}{8} & 1\\[1mm]
\frac{\pi}{2} & 0 & 0\\[1mm]
-\frac{\pi}{2} & \pi & 1
\end{bmatrix},
\qquad
\det D_{(\alpha,b,L)}F(0,0,0,\pi)=-\frac{3\pi^2}{16}\neq 0.
\end{equation}
In particular, $D_{(\alpha,b,L)}F(0,0,0,\pi)$ is invertible. 
\end{lemma}

\begin{proof}
At $\alpha=0$, the smooth extension is justified by the integral representation
\begin{equation}\label{eq:ideal:Phi_integral_rep}
\Phi(\alpha,\varepsilon,b,L)=\int_0^1 \partial_\alpha G(t\alpha,\varepsilon,b,L)\,dt,
\end{equation}
which follows from the fundamental theorem of calculus in the $\alpha$-variable and smoothness of $G$.

We compute the three rows of $D_{(\alpha,b,L)}F$ at the base point $(\alpha,\varepsilon,b,L)=(0,0,0,\pi)$.

\smallskip
\noindent\emph{Row 1: derivatives of $\Phi$.}
Differentiating \eqref{eq:ideal:Phi_integral_rep} under the integral sign gives, for $\xi\in\{b,L\}$,
\begin{equation}\label{eq:ideal:Phi_xi_formula}
\Phi_\xi(\alpha,\varepsilon,b,L)=\int_0^1 \partial_{\alpha\xi}G(t\alpha,\varepsilon,b,L)\,dt
\quad\Longrightarrow\quad
\Phi_\xi(0,\varepsilon,b,L)=\partial_{\alpha\xi}G(0,\varepsilon,b,L).
\end{equation}
In particular,
\begin{equation}\label{eq:ideal:Phi_L_from_G}
\Phi_L(0,0,0,\pi)=\partial_{\alpha L}G(0,0,0,\pi),
\qquad
\Phi_b(0,0,0,\pi)=\partial_{\alpha b}G(0,0,0,\pi).
\end{equation}

\smallskip
\emph{Computation of $\Phi_L(0,0,0,\pi)$.}
By Lemma~\ref{lem:ideal:Phi-alpha0}, for $(\varepsilon,L)$ near $(0,\pi)$,
\[
\Phi(0,\varepsilon,0,L)=-(\sin L+\varepsilon L),
\]
hence differentiating in $L$ and evaluating at $(\varepsilon,L)=(0,\pi)$ yields
\begin{equation}\label{eq:ideal:Phi_L_value_again}
\Phi_L(0,0,0,\pi)=-\cos\pi=1.
\end{equation}

\smallskip
\emph{Computation of $\Phi_b(0,0,0,\pi)$.}
We compute $\partial_{\alpha b}G(0,0,0,\pi)$ directly from the curvature-jet subsystem, as follows.
Along the base semicircle at $(\alpha,\varepsilon,b)=(0,0,0)$ we have
\[
k\equiv 1,\qquad \theta(s)=s,\qquad v:=k_s\equiv 0,\qquad w:=k_{ss}\equiv 0,\qquad u:=k_{sss}\equiv 0,
\]
and $G=V+U=v(L)+u(L)$.

\emph{(i) $b$-variation at $\alpha=0$.}
At $\alpha=0$ the jet equations reduce to the homogeneous linear system
\[
v_s=w,\qquad w_s=u,\qquad u_s=-w
\]
(the last identity is the $\alpha=0$ specialisation of $u_s+k(kw-\tfrac12v^2)=0$ with $k\equiv 1$).
Varying $b$ changes the initial condition $w(0)=q(0)$, so with $(v_b(0),w_b(0),u_b(0))=(0,1,0)$ we obtain
\begin{equation}\label{eq:ideal:b_variation}
w_b(s)=\cos s,\qquad u_b(s)=-\sin s,\qquad v_b(s)=\sin s.
\end{equation}
Consequently $k_b'(s)=v_b(s)$ with $k_b(0)=0$ gives $k_b(s)=1-\cos s$, and hence
$\theta_b'(s)=k_b(s)$ with $\theta_b(0)=0$ gives $\theta_b(s)=s-\sin s$.

\emph{(ii) $\alpha$-variation at $b=0$ (with $\varepsilon=0$).}
At $(\varepsilon,b)=(0,0)$, linearising the forcing in the homothetic equation yields (as in the proof of
Lemma~\ref{lem:ideal:Phi-alpha0})
\begin{equation}\label{eq:ideal:alpha_variation}
w_\alpha(s)=-\frac12 s\sin s,\qquad
v_\alpha(s)=-\frac12(\sin s-s\cos s),\qquad
u_\alpha(s)=-\frac12(\sin s+s\cos s),
\end{equation}
and integrating $k_\alpha'(s)=v_\alpha(s)$ with $k_\alpha(0)=0$ gives
\begin{equation}\label{eq:ideal:k_alpha}
k_\alpha(s)=\cos s-1+\frac12 s\sin s.
\end{equation}

\emph{(iii) Mixed $(\alpha,b)$-variation.}
Differentiate the $\alpha$-linearised equation
\begin{equation}\label{eq:ideal:alpha_lin_u_eq}
u_\alpha'+w_\alpha=-\cos\theta
\end{equation}
with respect to $b$ and evaluate along the base semicircle (so $\theta=s$).
The left-hand side becomes $u_{\alpha b}'+w_{\alpha b}$.
For the right-hand side, the $b$-dependence enters through $\theta=\theta(\alpha,b)$, hence
\[
\partial_b(-\cos\theta)=\sin s\,\theta_b(s).
\]
In addition, when deriving \eqref{eq:ideal:alpha_lin_u_eq} from
$u_s+k(kw-\tfrac12v^2)=-\alpha\cos\theta$, the product $k(kw-\tfrac12 v^2)$ contributes further
mixed terms upon differentiating in $b$ (because $k_b\not\equiv 0$). A direct differentiation gives
\begin{equation}\label{eq:ideal:mixed_u_eq}
u_{\alpha b}'+w_{\alpha b}
=
-\,2k_b w_\alpha \;-\;2k_\alpha w_b \;+\; v_b v_\alpha \;+\; \sin s\,\theta_b,
\end{equation}
with homogeneous initial data $v_{\alpha b}(0)=w_{\alpha b}(0)=u_{\alpha b}(0)=0$.
Together with
\[
v_{\alpha b}'=w_{\alpha b},\qquad w_{\alpha b}'=u_{\alpha b},
\]
set
\[
S:=v_{\alpha b}+u_{\alpha b}.
\]
Then using \eqref{eq:ideal:mixed_u_eq} we obtain the cancellation
\begin{equation}\label{eq:ideal:Sprime}
S'(s)=v_{\alpha b}'(s)+u_{\alpha b}'(s)=w_{\alpha b}(s)+\bigl(-w_{\alpha b}(s)+\mathrm{RHS}(s)\bigr)=\mathrm{RHS}(s),
\qquad S(0)=0,
\end{equation}
where $\mathrm{RHS}(s)$ denotes the right-hand side of \eqref{eq:ideal:mixed_u_eq}.

Substitute the explicit expressions
\eqref{eq:ideal:b_variation}, \eqref{eq:ideal:alpha_variation}, \eqref{eq:ideal:k_alpha},
and $\theta_b(s)=s-\sin s$, then simplify:
\begin{align*}
\mathrm{RHS}(s)
&=
-2(1-\cos s)\Bigl(-\frac12 s\sin s\Bigr)
-2\Bigl(\cos s-1+\frac12 s\sin s\Bigr)\cos s
\\&\qquad 
+(\sin s)\Bigl(-\frac12(\sin s-s\cos s)\Bigr)
+\sin s\,(s-\sin s).
\\&=2s\sin s-\frac34\,s\sin(2s)+2\cos s-\frac14\cos(2s)-\frac74.
\end{align*}
Integrating \eqref{eq:ideal:Sprime} from $0$ to $L$ gives
\begin{equation}\label{eq:ideal:S_generalL}
S(L)=\partial_{\alpha b}G(0,0,0,L)
=
-2L\cos L+\frac38 L\cos(2L)-\frac74 L +4\sin L-\frac{5}{16}\sin(2L).
\end{equation}
Evaluating at $L=\pi$ yields
\begin{equation}\label{eq:ideal:Phi_b_value_again}
\Phi_b(0,0,0,\pi)=\partial_{\alpha b}G(0,0,0,\pi)=S(\pi)
=\Bigl(2\pi+\frac{3\pi}{8}-\frac{7\pi}{4}\Bigr)=\frac{5\pi}{8}.
\end{equation}

Combining \eqref{eq:ideal:Phi_L_value_again} and \eqref{eq:ideal:Phi_b_value_again} gives the first row
$\bigl(\Phi_\alpha(0,0,0,\pi),\,\frac{5\pi}{8},\,1\bigr)$.

\smallskip
\noindent\emph{Row 2: derivatives of $U$.}
Recall $U(\alpha,\varepsilon,b,L)=M(L)$ and, by Lemma~\ref{lem:ideal:xyQ_consistency}, $M=k_{sss}=u$ on any interval where
the reconstruction is valid. Along the base semicircle $u\equiv 0$.

\emph{Derivative in $\alpha$.}
Along the base semicircle and with $\varepsilon=0$, the $\alpha$-variation $u_\alpha$ is given in
\eqref{eq:ideal:alpha_variation}, hence
\begin{equation}\label{eq:ideal:U_alpha_value_again}
U_\alpha(0,0,0,\pi)=u_\alpha(\pi)=-\frac12(\sin\pi+\pi\cos\pi)=\frac{\pi}{2}.
\end{equation}

\emph{Derivative in $b$.}
From \eqref{eq:ideal:b_variation}, $u_b(s)=-\sin s$, hence $u_b(\pi)=0$ and
\begin{equation}\label{eq:ideal:U_b_value_again}
U_b(0,0,0,\pi)=0.
\end{equation}

\emph{Derivative in $L$.}
Since $U$ is evaluation at $s=L$, we have $U_L(0,0,0,\pi)=u_s(\pi)$ along the base.
But $u\equiv 0$ on the base semicircle, hence $u_s(\pi)=0$ and
\begin{equation}\label{eq:ideal:U_L_value_again}
U_L(0,0,0,\pi)=0.
\end{equation}
This gives the second row $\bigl(\frac{\pi}{2},\,0,\,0\bigr)$.

\smallskip
\noindent\emph{Row 3: derivatives of $B$.}
At $\varepsilon=0$ we have $B(\alpha,0,b,L)=-\sin\theta(L)$, so at the base point
\begin{align*}
B_\xi(0,0,0,\pi)&=-\cos\theta(\pi)\,\theta_\xi(\pi)=\theta_\xi(\pi)
\qquad(\xi\in\{\alpha,b\}),
\qquad\\
B_L(0,0,0,\pi)&=-\cos\theta(\pi)\,\theta_s(\pi)=-\cos\pi\cdot k(\pi)=1,
\end{align*}
since $\theta(\pi)=\pi$ and $k(\pi)=1$ on the base.

\emph{Derivative in $\alpha$.}
Using $k_\alpha$ from \eqref{eq:ideal:k_alpha} and $\theta_\alpha'(s)=k_\alpha(s)$ with $\theta_\alpha(0)=0$,
\[
\theta_\alpha(\pi)=\int_0^\pi k_\alpha(s)\,ds
=\int_0^\pi(\cos s-1)\,ds+\frac12\int_0^\pi s\sin s\,ds
=0-\pi+\frac12\cdot\pi=-\frac{\pi}{2}.
\]
Thus
\begin{equation}\label{eq:ideal:B_alpha_value_again}
B_\alpha(0,0,0,\pi)=-\frac{\pi}{2}.
\end{equation}

\emph{Derivative in $b$.}
From the $b$-variation above, $k_b(s)=1-\cos s$, hence
\[
\theta_b(\pi)=\int_0^\pi k_b(s)\,ds=\int_0^\pi(1-\cos s)\,ds=\pi,
\]
and therefore
\begin{equation}\label{eq:ideal:B_b_value_again}
B_b(0,0,0,\pi)=\pi.
\end{equation}
This gives the third row $\bigl(-\frac{\pi}{2},\,\pi,\,1\bigr)$.

\smallskip
\noindent\emph{Determinant and invertibility.}
Collecting the three rows yields the matrix in \eqref{eq:ideal:Jacobian}.
Expanding the determinant along the second row (which has two zeros) gives
\[
\det D_{(\alpha,b,L)}F(0,0,0,\pi)
=\frac{\pi}{2}\,
\det\begin{bmatrix}
\frac{5\pi}{8} & 1\\[0.5mm]
\pi & 1
\end{bmatrix}
=\frac{\pi}{2}\Bigl(\frac{5\pi}{8}-\pi\Bigr)
=-\frac{3\pi^2}{16}\neq 0.
\]
Hence the Jacobian is invertible. 
\end{proof}

We thus conclude the existence of fundamental arcs.

\begin{proposition}\label{prop:ideal:fundamental-arc}
There exist $\varepsilon_0>0$ and unique $C^\infty$ functions
\[
\varepsilon\longmapsto \alpha(\varepsilon),\qquad
\varepsilon\longmapsto b(\varepsilon),\qquad
\varepsilon\longmapsto L(\varepsilon),
\qquad |\varepsilon|<\varepsilon_0,
\]
with $(\alpha(0),b(0),L(0))=(0,0,\pi)$ such that
\begin{equation}\label{eq:ideal:Fzero}
F\bigl(\alpha(\varepsilon),\varepsilon,b(\varepsilon),L(\varepsilon)\bigr)=0.
\end{equation}
In particular, for each $|\varepsilon|<\varepsilon_0$ the corresponding solution of
\eqref{eq:ideal:homothetic_xyQ_ode}-\eqref{eq:ideal:IC} is a smooth fundamental arc
in the sense of Definition~\ref{def:ideal:fundamental-arc}.
\end{proposition}

\begin{proof}
By Lemma~\ref{lem:ideal:base-semicircle}, $F(0,0,0,\pi)=0$, and by Lemma~\ref{lem:ideal:Jacobian}
the Jacobian $D_{(\alpha,b,L)}F(0,0,0,\pi)$ is invertible.  The Implicit Function Theorem gives the claim.
\end{proof}

We now show that the fundamental arc is not circular.

\begin{lemma}\label{lem:ideal:first-order}
Along the branch of solutions given by  Proposition~\ref{prop:ideal:fundamental-arc} one has
\begin{equation}\label{eq:ideal:first-order}
\alpha'(0)=0,\qquad b'(0)=-\frac{8}{3},\qquad L'(0)=\frac{8\pi}{3}.
\end{equation}
In particular, $b(\varepsilon)\neq 0$ for all sufficiently small $\varepsilon\neq 0$, so the corresponding
fundamental arc is not a circular arc.
\end{lemma}

\begin{proof}
Differentiate \eqref{eq:ideal:Fzero} at $\varepsilon=0$ to obtain the linear system
\[
D_{(\alpha,b,L)}F(0,0,0,\pi)
\begin{bmatrix}\alpha'(0)\\ b'(0)\\ L'(0)\end{bmatrix}
=-F_\varepsilon(0,0,0,\pi).
\]
By Lemma~\ref{lem:ideal:Phi-alpha0}, $\Phi_\varepsilon(0,0,0,\pi)=-\pi$, and at $\alpha=0$ the ODE is
$\varepsilon$-independent, hence $U_\varepsilon(0,0,0,\pi)=B_\varepsilon(0,0,0,\pi)=0$.
Solving the resulting $3\times 3$ system using \eqref{eq:ideal:Jacobian} yields \eqref{eq:ideal:first-order}.
Finally, $b(\varepsilon)=q(0)$, so $b(\varepsilon)\neq 0$ implies $k$ is not constant.
\end{proof}


\begin{figure}[t]
\centering

\begin{minipage}[t]{0.69\textwidth}
  \vspace{4pt}\centering

  \begin{subfigure}[t]{\linewidth}
    \centering
    \incfigl{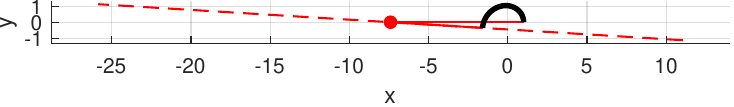}{0.95\linewidth}{8mm}{10mm}
    \caption{Fundamental arc.}
  \end{subfigure}

  \vspace{12mm}

  \begin{subfigure}[t]{\linewidth}
    \centering
    \incfigl{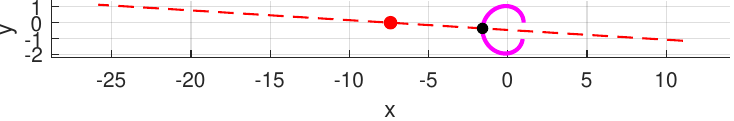}{0.95\linewidth}{8mm}{10mm}
    \caption{Doubled arc.}
  \end{subfigure}
\end{minipage}\hfill
\begin{minipage}[t]{0.30\textwidth}
  \vspace{0pt}\centering

  \begin{subfigure}[t]{\linewidth}
    \centering
    \incfig{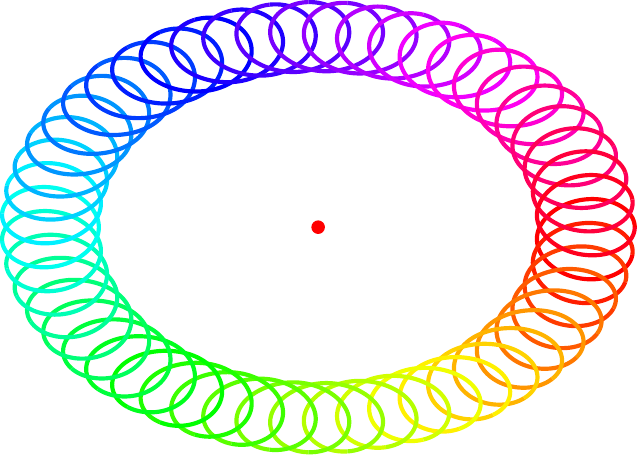}{\linewidth}{0}
    \caption{Closed curve (51 doubled arcs).}
  \end{subfigure}
\end{minipage}

\caption{Ideal-flow epicyclic shrinker with $\omega=50$ and $\varepsilon=0.13507$ (to five significant figures).
The closed curve is assembled from $51$ doubled arcs.}
\label{fig:ideal:50_51}
\end{figure}

\subsection{Terminal turning angle and abundance of rational angles}
\label{subsec:ideal:turning-angle}

Define the terminal tangent-angle
\begin{equation}\label{eq:ideal:ThetaDef}
\Theta(\varepsilon):=\theta\bigl(L(\varepsilon);\alpha(\varepsilon),\varepsilon,b(\varepsilon)\bigr).
\end{equation}

We require, as with curve diffusion flow, expansion of $\Theta$ to second-order.

\begin{lemma}\label{lem:ideal:Theta-expansion}
Along the branch of solutions given by Proposition \ref{prop:ideal:fundamental-arc},
\begin{equation}\label{eq:ideal:Theta-expansion}
\Theta(\varepsilon)=\pi-\frac{4\pi}{3}\,\varepsilon^2+O(\varepsilon^3)
\qquad(\varepsilon\to 0).
\end{equation}
In particular, $\Theta$ is continuous and non-constant in a neighbourhood of $\varepsilon=0$.
\end{lemma}

\begin{proof}
Along the branch from Proposition~\ref{prop:ideal:fundamental-arc} we set
\[
\Theta(\varepsilon):=\theta\bigl(L(\varepsilon);\alpha(\varepsilon),\varepsilon,b(\varepsilon)\bigr),
\]
where $(\alpha(\varepsilon),b(\varepsilon),L(\varepsilon))$ is the unique smooth solution curve with
$(\alpha(0),b(0),L(0))=(0,0,\pi)$.

Differentiate $\Theta(\varepsilon)=\theta(L(\varepsilon);\alpha(\varepsilon),\varepsilon,b(\varepsilon))$ at $\varepsilon=0$:
\begin{equation}\label{eq:ideal:Theta_prime_chain}
\Theta'(0)
=\theta_\varepsilon+\theta_\alpha\,\alpha'(0)+\theta_b\,b'(0)+\theta_L\,L'(0),
\end{equation}
where all partial derivatives are evaluated at $(\alpha,\varepsilon,b,L)=(0,0,0,\pi)$.
At $\alpha=0$ the ODE system is independent of $\varepsilon$, hence
\begin{equation}\label{eq:ideal:theta_eps_zero_again}
\theta_\varepsilon(0,0,0,\pi)=0.
\end{equation}
Moreover, the first-order expansion along the branch gives
\begin{equation}\label{eq:ideal:first_order_values_again}
\alpha'(0)=0,\qquad b'(0)=-\frac{8}{3},\qquad L'(0)=\frac{8\pi}{3}.
\end{equation}
Finally, along the base semicircle $k\equiv 1$ so $\theta_L(0,0,0,\pi)=\theta_s(\pi)=k(\pi)=1$, and from the $b$-variation
(computed in \eqref{eq:ideal:b_variation}) we have $\theta_b(\pi)=\pi$. Substituting into \eqref{eq:ideal:Theta_prime_chain} yields
\[
\Theta'(0)=\theta_b(\pi)\,b'(0)+\theta_L(\pi)\,L'(0)=\pi\Bigl(-\frac{8}{3}\Bigr)+1\cdot\frac{8\pi}{3}=0.
\]
Thus
\begin{equation}\label{eq:ideal:Theta_prime_zero}
\Theta'(0)=0.
\end{equation}

Along the branch, the seam condition $B(\alpha(\varepsilon),\varepsilon,b(\varepsilon),L(\varepsilon))=0$ reads
\begin{equation}\label{eq:ideal:seam_B0}
0=-\sin\Theta(\varepsilon)
+\varepsilon\,J(\varepsilon),
\qquad
J(\varepsilon):=
\bigl(y\cos\theta-x\sin\theta\bigr)\big|_{s=L(\varepsilon)}
=\bigl(\gamma\cdot T\bigr)\big|_{s=L(\varepsilon)}.
\end{equation}
At $\varepsilon=0$ we are on the base semicircle, hence $\Theta(0)=\pi$ and $J(0)=0$ (indeed $\gamma\cdot T\equiv 0$ on the circle).
Set $\delta(\varepsilon):=\Theta(\varepsilon)-\pi$. Then \eqref{eq:ideal:Theta_prime_zero} implies $\delta(\varepsilon)=O(\varepsilon^2)$.
Using $-\sin(\pi+\delta)=\sin\delta=\delta+O(\delta^3)$, the identity \eqref{eq:ideal:seam_B0} becomes
\begin{equation}\label{eq:ideal:delta_eq}
\delta(\varepsilon)=-\,\varepsilon\,J(\varepsilon)+O(\varepsilon^3).
\end{equation}
Since $J(0)=0$, we may write $J(\varepsilon)=J'(0)\varepsilon+O(\varepsilon^2)$, and substituting into \eqref{eq:ideal:delta_eq} yields
\begin{equation}\label{eq:ideal:delta_from_Jprime}
\delta(\varepsilon)=-J'(0)\,\varepsilon^2+O(\varepsilon^3).
\end{equation}
Therefore, the claimed quadratic term in $\Theta$ reduces to computing $J'(0)$.

Differentiate $J(\varepsilon)$ along the branch at $\varepsilon=0$:
\begin{equation}\label{eq:ideal:Jprime_chain}
J'(0)=J_\varepsilon+J_\alpha\,\alpha'(0)+J_b\,b'(0)+J_L\,L'(0),
\end{equation}
again evaluated at $(\alpha,\varepsilon,b,L)=(0,0,0,\pi)$.
As noted above, at $\alpha=0$ the ODE is independent of $\varepsilon$, hence
\begin{equation}\label{eq:ideal:J_eps_zero}
J_\varepsilon(0,0,0,\pi)=0.
\end{equation}
Also $\alpha'(0)=0$ by \eqref{eq:ideal:first_order_values_again}, so
\begin{equation}\label{eq:ideal:Jprime_reduced_again}
J'(0)=J_b(0,0,0,\pi)\,b'(0)+J_L(0,0,0,\pi)\,L'(0).
\end{equation}

We now compute $J_b$ and $J_L$ at the base semicircle.
Along the base semicircle, $\gamma(s)=(\cos s,\sin s)$, $\theta(s)=s$, and $T(s)=(-\sin s,\cos s)$, so
$J(s)=\gamma(s)\cdot T(s)\equiv 0$. Hence
\begin{equation}\label{eq:ideal:J_L_zero_again}
J_L(0,0,0,\pi)=\partial_L\bigl(J(L)\bigr)\Big|_{L=\pi}=J_s(\pi)=0.
\end{equation}

For $J_b$, write at fixed $s$:
\begin{equation}\label{eq:ideal:J_b_split}
J_b=\gamma_b\cdot T+\gamma\cdot T_b.
\end{equation}
Since $T=(-\sin\theta,\cos\theta)$, we have $T_b=(-\cos\theta\,\theta_b,-\sin\theta\,\theta_b)$.
Evaluating at $s=\pi$ on the base semicircle gives $\gamma(\pi)=(-1,0)$, $T(\pi)=(0,-1)$, $\theta(\pi)=\pi$, hence
\begin{equation}\label{eq:ideal:J_b_pi_reduce_again}
J_b(\pi)=\gamma_b(\pi)\cdot T(\pi)+\gamma(\pi)\cdot T_b(\pi)
=-\,y_b(\pi)\;-\;\theta_b(\pi).
\end{equation}
We have already shown that $\theta_b(\pi)=\pi$. It remains to compute $y_b(\pi)$.
Differentiate $y_s=\cos\theta$ in $b$ at the base: $y_b'=-\sin s\,\theta_b(s)$ with $y_b(0)=0$.
Using $\theta_b(s)=s-\sin s$ yields
\begin{align*}
y_b(\pi)&=\int_0^\pi -\sin s\,(s-\sin s)\,ds
\\&=-\int_0^\pi s\sin s\,ds+\int_0^\pi \sin^2 s\,ds
=-\pi+\frac{\pi}{2}=-\frac{\pi}{2}
.
\end{align*}
Substituting into \eqref{eq:ideal:J_b_pi_reduce_again} gives
\begin{equation}\label{eq:ideal:J_b_value_again}
J_b(0,0,0,\pi)=J_b(\pi)= -\Bigl(-\frac{\pi}{2}\Bigr)-\pi=-\frac{\pi}{2}.
\end{equation}

Combining \eqref{eq:ideal:Jprime_reduced_again}, \eqref{eq:ideal:J_L_zero_again}, \eqref{eq:ideal:J_b_value_again} with
$b'(0)=-\frac{8}{3}$ and $L'(0)=\frac{8\pi}{3}$ yields
\begin{equation}\label{eq:ideal:Jprime_value_again}
J'(0)=\Bigl(-\frac{\pi}{2}\Bigr)\Bigl(-\frac{8}{3}\Bigr)=\frac{4\pi}{3}.
\end{equation}

Insert \eqref{eq:ideal:Jprime_value_again} into \eqref{eq:ideal:delta_from_Jprime}:
\[
\Theta(\varepsilon)
=\pi+\delta(\varepsilon)
=\pi-\frac{4\pi}{3}\,\varepsilon^2+O(\varepsilon^3),
\]
which is exactly \eqref{eq:ideal:Theta-expansion}. Continuity of $\Theta$ is immediate from smooth dependence of the ODE solutions
and of $(\alpha(\varepsilon),b(\varepsilon),L(\varepsilon))$ on $\varepsilon$. Since the quadratic coefficient is nonzero, $\Theta$
is non-constant on any neighbourhood of $\varepsilon=0$.
\end{proof}

We thus conclude, as before, an abundance of angles.

\begin{corollary}\label{cor:ideal:rational-angles}
There exist $\varepsilon_1\in(0,\varepsilon_0)$ and an open interval $I\subset\R$ such that
$\Theta\bigl((-\varepsilon_1,\varepsilon_1)\bigr)\supset I$.
Consequently, there are infinitely many $\varepsilon\in(-\varepsilon_1,\varepsilon_1)$ such that
\[
\frac{\Theta(\varepsilon)}{\pi}\in\Q,
\qquad\text{equivalently }\quad
\Theta(\varepsilon)=\frac{p}{q}\,\pi \ \text{ for coprime integers }p,q\ge 1.
\]
\end{corollary}

\subsection{Dihedral gluing and proof of Theorem~\ref{thm:intro:if-epi}}
\label{subsec:ideal:gluing}

The gluing argument is identical to Section~\ref{subsec:cdf:dihedral-gluing}, once one notes that the ideal homothetic equation
is invariant under Euclidean isometries fixing $c_\varepsilon$, and that the seam conditions are now
\[
(\gamma-c_\varepsilon)\cdot T=0,\qquad k_s=0,\qquad k_{sss}=0
\]
at each seam point (equivalently $B=0$, $V=0$, $U=0$).

\begin{lemma}\label{lem:ideal:reflection-principle}
Let $\gamma:[0,L]\to\R^2$ be an arclength-parametrised solution of the ideal-flow homothetic ODE on $(0,L)$.
Assume $\gamma(L)$ lies on a line $\ell$ through $c_\varepsilon$, $\gamma$ meets $\ell$ orthogonally at $s=L$, and
\[
k_s(L)=0,\qquad k_{sss}(L)=0.
\]
Then reflecting $\gamma$ across $\ell$ and reversing arclength glues to a $C^\infty$ solution on $[0,2L]$.
The same holds at $s=0$.
\end{lemma}

\begin{proof}
This is the same even/odd continuation argument as Lemma~\ref{lem:reflection-principle-jellyfish}, with the additional jet condition $k_{sss}(L)=0$ ensuring smooth matching
for the higher-order system. Details are  omitted.
\end{proof}

Dihedral closure and geometric distinctness follows using a similar approach to earlier; we summarise the results in the following.

\begin{proposition}\label{prop:ideal:dihedral-closure}
Let $\varepsilon$ be such that $\Theta(\varepsilon)=\frac{p}{q}\pi$ (coprime $p,q$), and let $\gamma:[0,L]\to\R^2$ be the
fundamental arc from Proposition~\ref{prop:ideal:fundamental-arc}.
Then by doubling via Lemma~\ref{lem:ideal:reflection-principle} and concatenating $q$ rotated copies about $c_\varepsilon$
 one obtains a $C^\infty$ closed immersed curve solving the ideal-flow
homothetic equation pointwise.  The resulting closed curve has dihedral symmetry of order $q$.
\end{proposition}

\begin{proof}
The details are similar to the proof of Proposition~\ref{prop:cdf:dihedral-closure}, 
 and are omitted.
\end{proof}

\subsection{Geometric distinctness}
\label{subsec:ideal:distinctness}

The distinctness argument is the same as for the curve-diffusion epicyclic shrinkers:
the pair $(p,q)$ is preserved by similarity (rotation index $p$ and maximal dihedral symmetry order $q$),
and maximality is enforced by excluding ``interior radial seams'' on a fundamental arc.

\begin{lemma}\label{lem:ideal:turning-number}
Let $\Gamma$ be the closed curve produced by Proposition~\ref{prop:ideal:dihedral-closure} with $\Theta=\frac{p}{q}\pi$.
Then the total turning of the tangent is $2q\Theta=2p\pi$, hence the rotation index of $\Gamma$ equals $p$.
\end{lemma}

\begin{proof}
Similar to Lemma~\ref{lem:cdf:turning-number}.
\end{proof}

\begin{lemma}\label{lem:ideal:no-interior-seams}
For $|\varepsilon|$ sufficiently small, the seam function
\[
s\longmapsto -\sin\theta(s)+\varepsilon\bigl(y(s)\cos\theta(s)-x(s)\sin\theta(s)\bigr)
\]
along the fundamental arc from Proposition~\ref{prop:ideal:fundamental-arc} vanishes only at $s=0$ and $s=L(\varepsilon)$.
Equivalently, $(\gamma(s)-c_\varepsilon)\cdot T(s)\neq 0$ for all $s\in(0,L(\varepsilon))$.
\end{lemma}

\begin{proof}
The proof is similar to Lemma~\ref{lem:cdf:no-interior-seams}, using smooth dependence on parameters and the fact that
for the base semicircle the seam function is $-\sin s$ with exactly two simple zeros on $[0,\pi]$.
\end{proof}

\begin{lemma}\label{lem:ideal:maximal-dihedral}
Let $\Gamma$ be a closed curve produced by Proposition~\ref{prop:ideal:dihedral-closure} with $\Theta=\frac{p}{q}\pi$.
Then the full symmetry group of $\Gamma$ is exactly $D_q$.
\end{lemma}

\begin{proof}
Similar to Lemma~\ref{lem:cdf:no-extra-axes}, with Lemma~\ref{lem:ideal:no-interior-seams} excluding any extra reflection axis.
\end{proof}

\begin{corollary}\label{cor:ideal:non-equivalence}
Let $\Gamma_{p/q}$ and $\Gamma_{\hat p/\hat q}$ be closed ideal-flow epicyclic expanders produced by
Proposition~\ref{prop:ideal:dihedral-closure} from coprime pairs $(p,q)$ and $(\hat p,\hat q)$.
If $(p,q)\neq(\hat p,\hat q)$, then $\Gamma_{p/q}$ is not equivalent to $\Gamma_{\hat p/\hat q}$ under any similarity.
\end{corollary}

\begin{proof}
As in Corollary~\ref{cor:cdf:non-equivalence}: similarities preserve rotation index $p$ (Lemma~\ref{lem:ideal:turning-number})
and conjugate the full symmetry group (Lemma~\ref{lem:ideal:maximal-dihedral}).
\end{proof}

\begin{proof}[Proof of Theorem~\ref{thm:intro:if-epi}]
Choose infinitely many $\varepsilon$ with $\Theta(\varepsilon)/\pi\in\Q$ given by Corollary~\ref{cor:ideal:rational-angles}.
For each such $\varepsilon$, Proposition~\ref{prop:ideal:dihedral-closure} yields a smooth closed homothetic solution,
hence an epicyclic expander for the ideal flow.  Geometric distinctness follows from Corollary \ref{cor:ideal:non-equivalence}.
\end{proof}

\bibliography{jellepi}

\end{document}